\documentclass[a4paper,10pt]{article}
\usepackage[utf8x]{inputenc}
\pdfoutput=1
\usepackage{nccfoots}

\usepackage{times}
\usepackage{geometry}
\usepackage[T1]{fontenc}
\usepackage{color}
\usepackage{amsmath}
\usepackage{amssymb}
\usepackage{array}
\usepackage{amsthm}
\usepackage{graphicx}
\usepackage{hyperref}
\usepackage{setspace}
\usepackage{stmaryrd}
\usepackage{marginnote}
\usepackage[mathscr]{euscript}

\theoremstyle{plain}
\newtheorem{thm}{Theorem}[section]

\newtheorem{cor}[thm]{Corollary}
\newtheorem{lem}[thm]{Lemma}

\theoremstyle{definition}
\newtheorem{defn}[thm]{Definition}

\theoremstyle{remark}
\newtheorem{rem}[thm]{Remark}

\theoremstyle{plain}

\newcommand{\T}{\mathrm{T}}

\newcommand{\R}{\mathbb{R}}

\newcommand{\N}{\mathbb{N}}

\newcommand{\identity}{\mathrm{id}}

\newcommand{\ric}{\mathrm{Ric}}
\newcommand{\trace}{\mathrm{tr}}
\newcommand{\kernel}{\mathrm{ker}}

\newcommand{\volume}{\mathrm{vol}}

\newcommand{\Diff}{\mathrm{Diff}}
\newcommand{\image}{\mathrm{im}}
\newcommand{\arcsinh}{\mathrm{arcsinh}}

\newcommand{\gradient}{\mathrm{grad}}

\newcommand{\diver}{\mathrm{div}}
\newcommand{\brk}{\text{ }}

\newcommand{\abg}[1]{|#1|_{g}}

\newcommand{\go}[1]{\/^{\scriptscriptstyle{(#1)}}\!\tilde g}

\newcommand{\eq}[1]{\begin{equation}#1\end{equation}}

\newcommand{\alg}[1]{\begin{aligned}#1\end{aligned}}

\newcommand{\bc}{\begin{cases}}
\newcommand{\ec}{\end{cases}}

\newcommand{\Chr}[3]{\Gamma^{#1}_{#2 #3}}
\newcommand{\Chrh}[3]{\hat{\Gamma}^{#1}_{#2 #3}}

\newcommand{\p}[1]{\partial_{#1}}

\newcommand{\Ab}[1]{\|#1\|}

\newcommand{\Abk}[2]{\|#1\|_{H^{#2}}}
\newcommand{\Abkg}[2]{\|#1\|_{H^{#2}(g)}}

\newcommand{\ga}{\gamma}

\newcommand{\Si}{\Sigma}
\newcommand{\La}{\Lambda}

\newcommand{\De}{\Delta}

\newcommand{\na}{\nabla}

\newcommand{\mcl}[1]{\mathcal{#1}}
\newcommand{\mcr}[1]{\mathscr{#1}}

\renewcommand{\title}[1]{{\bfseries #1}\par}
\renewcommand{\author}[1]{\medskip{#1}\par\smallskip}

\numberwithin{equation}{section}

\begin{document}
\begin{center}
\title{\LARGE Stable fixed points of the Einstein flow with positive cosmological constant}
\vspace{3mm}
\author{\large David Fajman and Klaus Kröncke}
\Footnotetext{}{1991 \emph{Mathematics Subject Classification.} 58J45,53C25,83C05.}
\Footnotetext{}{\emph{Key words and phrases.} Nonlinear Stability, General Relativity, Einstein metrics, Einstein flow.}

\vspace{1mm}
\end{center}

\begin{center}
April 30, 2018
\end{center}
\vspace{2mm}

\begin{abstract}
We prove nonlinear stability for a large class of solutions to the Einstein equations with a positive cosmological constant and
compact spatial topology in arbitrary dimensions, where the spatial metric is Einstein with either positive or negative Einstein constant. The proof uses the CMC Einstein flow and
stability follows by an energy argument. We prove in addition that the development of non-CMC initial data close to the background contains a CMC hypersurface,
which in turn implies that stability holds for arbitrary perturbations. Furthermore, we construct a one-parameter family of initial data such that above a critical parameter value the corresponding development is future and past incomplete.
\end{abstract}
 
%


\section{Introduction}
Understanding the long time behavior of the Einstein flow and the global geometry of the resulting spacetimes has been a major field of interest in General Relativity for the past 30 years. A particularly successful area concerns the nonlinear stability problem for explicit solutions to Einstein's field equations. The first general results are due to Friedrich \cite{Fr86} for deSitter space-time and Christodoulou-Klainerman \cite{ChKl93} for the Minkowski space-time. Since then several results of a similar nature for different backgrounds have been established. In this paper we focus on the Einstein-flow with a positive cosmological constant.

\subsection{The positive cosmological constant}
The late time asymptotics of homogeneous cosmological models in the presence of a cosmological constant have been first analyzed by Wald \cite{Wa83}. Following Friedrich's work on the stability problem \cite{Fr86,Fr86-1} for the 3+1-dimensional case, Anderson generalized the stabilty result to asymptotically de Sitter space of arbitrary even dimension \cite{An05}. Later, Ringstr\"om was able to find conditions on the initial data, such that the global evolution problem can be localized to a coordinate neighborhood \cite{Ri08}. Using this local result he showed stability and future completeness for large classes of initial data on arbitrary spatial topologies. This implies that the spatial topology itself cannot be deduced from the long time behavior of the Einstein flow in the presence of a positive cosmological constant. A result of this nature had also been established by Friedrich in \cite{Fr86-1}. Ringstr\"om's results hold for the Einstein-scalar field system and he later generalized them to the Einstein-Vlasov-scalar field system in \cite{Ri13}. Similar results have been obtained by Svedberg for the Einstein-Maxwell system \cite{Sv11}.\\
In the case of the Einstein-Euler system questions of long-time existence are complicated by the likely appearance of shocks. However, in the presence of a positive cosmological constant it has been shown by Rodnianski and Speck that the accelerated expansion is sufficiently strong to avoid shock formation in the non-vacuum setting for the irrotational Einstein-Euler system \cite{RoSp13} and by Speck for the general Einstein-Euler system \cite{Sp12}.

\subsection{The CMC-Einstein Flow on compact manifolds}
In the study of nonlinear stability of expanding solutions to the vacuum Einstein-flow with vanishing cosmological constant CMC (constant mean curvature) foliations have been proven to be very beneficial. The study of the CMC Einstein flow was initiated by the work of Andersson, Moncrief and Tromba on the global existence of CMC foliations of vacuum solutions of the Einstein equations in 2+1 dimensions \cite{AnMoTr97}. Fischer and Moncrief \cite{FiMo01,FiMo02} studied the Einstein flow in CMC gauge for the higher dimensional case, which eventually led to the proof of stability for FLRW(Friedmann-Lema\^{\i}tre-Robertson-Walker) type solutions in 3+1 dimensions by Andersson and Moncrief \cite{AnMo04} and finally to the stability of a large class of spatial Einstein geometries of negative scalar curvature in arbitrary dimensions by the same authors in 2011 \cite{AMo11}. The proof in \cite{AMo11} is based on a carefully adjusted energy argument which shows the asymptotic convergence of the perturbed solution to the spatial Einstein metric.\\
The motivation for the present paper is the study of the CMC Einstein flow with a positive cosmological constant, which has so far only been considered for the 2+1 dimensional case by Andersson, Moncrief and Tromba in \cite{AnMoTr97}. 
\subsection{Main Results}

Our main result is the nonlinear stability for a large class of solutions to the Einstein equation with positive cosmological constant.
The background solutions are homogeneous model solutions for the CMC-Einstein flow with positive cosmological constant where the spatial metric is an Einstein metric with positive or negative Einstein constant on a compact manifold with arbitary dimension $n\geq2$. The main theorem in the case of a negative Einstein constant is the following.
\begin{thm}\label{main-thm-neg}
Let $M$ be a smooth compact n-dimensional manifold ($n\geq 2$) without boundary and $\gamma$ be an Einstein metric satisfying ~$\ric(\ga)=-(n-1)\ga$. Then for $s>n/2+2$, $s'>n/2+s$ and $\varepsilon>0$ there exists a $\delta(\varepsilon)>0$ s.t.~for initial data $(g_0,k_0)$ satisfying
\eq{
\left\| g_0-\gamma\right\|_{H^{s'}}+\left\|k_0+\sqrt{2}\gamma\right\|_{H^{s'-1}}<\delta
}  
its maximal globally hyperbolic development under the Einstein flow with positive cosmological constant $\Lambda=\frac{n(n-1)}{2}$  can be foliated by CMC-hypersurfaces $M_t$, $t\in [\arcsinh(1),\infty)$ such that the induced metrics $g_t$ satisfy
\eq{
\left\|\sinh^{-2}(t) g_t-\gamma\right\|_{H^{s}}<\varepsilon.
}  
In particular, all corresponding homogeneous solutions are orbitally stable and the future developments of small perturbations are future geodesically complete.
\end{thm}
\noindent In the case of positive Einstein constant, the main theorem is the following.
\begin{thm}\label{main-thm}
Let $M$ be a smooth compact n-dimensional manifold ($n\geq 2$) without boundary and $\gamma$ be an Einstein metric satisfying ~$\ric(\ga)=(n-1)\ga$ which does not admit Killing vector fields and such that $-2(n-1)$ is not an eigenvalue of the Laplacian. Then for $s>n/2+2$, $s'>n/2+s$ and $\varepsilon>0$ there exists a $\delta(\varepsilon)>0$ s.t.~for initial data $(g_0,k_0)$ satisfying
\eq{
\left\| g_0-\gamma\right\|_{H^{s'}}+\left\|k_0\right\|_{H^{s'-1}}<\delta
}  
its maximal globally hyperbolic development under the Einstein flow with positive cosmological constant $\Lambda=\frac{n(n-1)}{2}$  can be globally foliated by CMC-hypersurfaces $M_t$, $t\in\R$ such that the induced metrics $g_t$ satisfy
\eq{
\left\|\cosh^{-2}(t) g_t-\gamma\right\|_{H^{s}}<\varepsilon.
}  
In particular, all corresponding homogeneous solutions are orbitally stable and the future-  and past developments of small perturbations are future- and past geodesically complete, respectively.
\end{thm}
Note that we do not assume that the initial data satisfies the CMC constraints.
In fact, we show that the maximal development of a small but arbitrary perturbation contains a CMC hypersurface so that we can start the CMC Einstein flow with the initial data induced on that hypersurface. During this process, we loose regularity which requires the initial data to lie in a small neighbourhood of higher regularity.

The idea of the proof is based on an energy argument, which makes use of the elliptic hyperbolic structure of the Einstein flow in CMCSH gauge (\emph{constant mean curvature spatial harmonic gauge}) and which is inspired by the ideas of Andersson and Moncrief in \cite{AMo11}. The presence of the positive cosmological constant yields a specific asymptotic hierarchy of the terms appearing in the CMCSH equations, which needs to be taken into account by choosing an appropriate rescaling of the evolving geometry. The eventual energy estimate does not contain a decay inducing negative term on the right hand side as exploited by a correction mechanism in \cite{AMo11}. Here, another crucial observation allows to obtain a sufficient energy estimate in the small data setting. The idea is as follows. The energy is essentially the sum of a geometric Sobolev norm of the difference between the rescaled spatial metric and the Einstein background geometry and that of a geometric Sobolev norm of the trace free part $\Si$ of the rescaled second fundamental form. The Sobolev norms are defined w.r.t.~a Laplace-type operator on tensors corresponding to the background geometry and the perturbed spatial metric. Straightforward energy estimates for this norm contain a negative term on the right hand side, which is a multiple of the Sobolev norm of $\Si$. This term on its own cannot be exploited to gain additional decay for the energy - also a correction mechanism fails in this case. However, a careful analysis of all additional perturbation terms on the right hand side of the energy estimate allows to use this term to absorb all perturbation terms with insufficient decay properties under the assumption that we are in a small data scenario. We obtain a strong energy estimate with an exponentially decaying coefficient under the smallness condition. Global existence and stability then follow by a bootstrap argument. Using boundedness for the highest order of regularity an isolated energy estimate for the second fundamental form in lower regularity is then obtained which yields improved decay properties for this norm. In turn, taking the rescaling into account, the asymptotics of the solution imply the desired completeness. \\
\noindent The main theorem which we prove partly contains results of previous papers mentioned above. In \cite{Ri08} Ringstr\"om proves nonlinear stability of the same background solutions in the case of positive curvature in dimensions $n\geq3$ and that of negative curvature in dimension 3. However, for these solutions the existence of CMC foliations has not been proven. A major advantage of our proof is that it works in arbitrary dimensions and does not rely on local coordinates. Its key feature is that it reduces the stability problem to a single energy estimate (given in lemma \ref{lem : en-est}) and thereby identifies the essence of the stability mechanism for those solutions. The method also appears sufficiently robust to generalize the result to the non-vacuum setting.
%

Besides, we construct a one-parameter family of initial data, such that there is a critical parameter value above which the corresponding future and past development recollapses, while for smaller values the developement is future and past complete. This family illustrates the variety of scenarios which might occur despite the presence of a cosmological constant. In particular, the threshold solution between both regimes is a new example for an unstable solution to the Einstein equations.
\subsection{Overview of the paper}
In section \ref{sec : 2} we introduce our notations and recall the equations for the Einstein flow in CMCSH gauge with positive cosmological constant.
 In subsection \ref{ssec : 24} we construct homogeneous solutions for the cases of positive and negative spatial Einstein manifolds. In subsection \ref{ssec : 26} we prove the existence of CMC surfaces in the MGHD (maximal globally hyperbolic development) of non-CMC initial data and also the generality of the spatial harmonicity condition. In section \ref{sec : 4} we construct the family of initial data containing both data with expanding and recollapsing asymptotics. Section \ref{sec : 6} contains the proof of the main theorem. In subsection \ref{ssec : 25} we address the problem of local existence for the data we consider. The proof of the main theorem is divided into the analysis of the elliptic system in subsection \ref{ssec : ell}, the main energy estimate in subsection \ref{ssec : enest}, an improved decay estimate in subsection \ref{ssec : imprdec} and the bootstrap argument which establishes global existence in subsection \ref{ssec : proof}.

\subsection*{Acknowledgements}
Parts of this work have been accomplished while the authors enjoyed the hospitality
of the Erwin-Schr\"odinger-Institute in Vienna. We thank Bobby Beig and Piotr Chru\'{s}ciel for helpful comments on the manuscript. D.~F.~thanks the Mathematics Faculty of the University of Regensburg for their hospitality and support during a visit. K.~K.~thanks the SFB 1085, funded by the Deutsche Forschungsgemeinschaft for financial support.


\subsection{Notations and Conventions}\label{sec:nac} We collect all relevant further notations used in this paper in the following.
We define the scalar product of two tensors $u,v$ w.r.t. $\ga$ by
\eq{
\langle u,v\rangle\equiv u_{ij}v_{kl}\ga^{ik}\ga^{jl}.
} 
In addition, we define the mixed $L^2$-scalar product by
\eq{
(u,v)_{L^{2}(g,\ga)}\equiv \int_M \langle u,v\rangle \mu_g.
} 
The corresponding $L^2$-scalar product where also the volume form is taken w.r.t.~$\ga$ is denoted by $(.~,~.)_\ga$.\\
We denote the standard Sobolev spaces on $M$ by $H^s$, where the norm is defined w.r.t.~a fixed metric on $M$. We do not distiguish the notation for Sobolev spaces of different types of tensors. We denote the $H^s$-norm by $\Abk{.}s$. The ball in the $H^s$ topology of radius $\varepsilon>0$ around a tensor $h$ is denoted by $\mcl B^s_{\varepsilon}(h)$. The covariant derivative of a certain metric $g$ is denoted by $\na[g]$. If it is clear from the context, the reference to the metric is suppressed.
The curvature tensor is defined with the sign convention such that
$R_{ijkl}X^k=\nabla_j\nabla_iX_l-\nabla_i\nabla_jX_l$. The Laplacians are defined with the sign convention such that all eigenvalues are nonpositive.

\section{The CMC-Einstein flow}\label{sec : 2}

\subsection{Background solutions}
Throughout the paper, we put the cosmological constant to $\Lambda=\frac{n(n-1)}{2}$ where $n$ is the spatial dimension.
Then Einstein's equations are equivalent to $\mathrm{Ric}_{\tilde g}=n \tilde g$. 
Assume that $(M,\gamma)$ is a compact Riemannian Einstein manifold with negative scalar curvature $R(\gamma)=-n(n-1)$. Then, the metric 
\eq{\label{model1}
\tilde{\gamma}=-dt^2+\sinh^2(t)\gamma
}
solves Einsteins equation on $(0,\infty)\times M$. This solution is future geodesically complete and the mean curvature of the slices $M_t=\left\{t\right\}\times M$ is 
$\tau(t)=-n\frac{\cosh(t)}{\sinh(t)}$
which is strictly monotonically increasing on $(0,\infty)$ and tends from $-\infty$ as $t\to 0$ to $-n$ as $t\to\infty$. Introducing $\tau$ as a new time variable, $\tilde{\gamma}$ is a solution of the CMCSH flow we will introduce in the next section. Similarly, if $\gamma$ is Einstein but of positive scalar curvature $R(\gamma)=n(n-1)$,
\eq{\label{model2}
\tilde{\gamma}=-dt^2+\cosh^2(t)\gamma
}
solves the Einstein equation on $\R \times M$
which is future and past geodesically complete. If $(M,\gamma)$ is the sphere, we recover the de-Sitter metric and therefore these models are called generalized de-Sitter spaces. The mean curvature of $M_t$ is 
$\tau(t)=-n\frac{\sinh(t)}{\cosh(t)}$ which is strictly monotonically decreasing. We have $\lim_{t\to\infty}\tau(t)= -n$ and
$\lim_{t\to-\infty}\tau(t)= n$. The metric $\tilde{\gamma}$ cannot be regarded as a solution of the CMCSH flow, but as a solution of the reversed CMCSH flow if we put the time variable to $-\tau$.
\begin{rem}
If $\gamma$ is a Ricci-flat metric, $\tilde{\gamma}=-dt^2+e^{2t}\gamma$
also solves the Einstein equation. Unfortunately, the mean curvature of any slice $M_t$ is constantly $-n$ and therefore, we cannot handle this solution with our gauge conditions.
\end{rem}

\subsection{ADM Einstein equations}
We consider a space-time of the form $\R\times M$, where $M$ is a smooth compact $n$-dimensional manifold without boundary. For the Lorentzian metric
we choose the ADM-Ansatz 
\eq{
\go {n+1}= -N^2dt\otimes dt+g_{ij}(dx^i+X^i dt)\otimes(dx^j+X^j dt),
}
where $\tilde{g}=(N,X,g)$ denote the lapse function, shift vector field and the spatial metric, respectively. The Einstein equations in CMCSH gauge,
\eq{\alg{
&\mathrm{tr}_g k=: \tau=t\\
&g^{ij}(\Chr kij-\Chrh kij)=: V^k=0,
}}
where $\Chr kij$, $\Chrh kij$ denote the Christoffel symbols w.r.t~$g$ and $\ga$, respectively, with positive cosmological constant $\La=\frac{n(n-1)}{2}$ read 
\eq{\label{CMC}
\alg{
R(g)-\abg{\Si}^2+\tau^2\left(\frac{n-1}{n}\right)&=n(n-1)\\
\na^i\Si_{ij}&=0\\
\p tg_{ij}&=-2N(\Si_{ij}+\tau/ng_{ij})+\mcr{L}_{X}g_{ij}\\
\p t\Si_{ij}&=N(R_{ij}+\tau \Si_{ij}-2\Si_{ik}\Si_j^k+(\tau^2/n-n)g_{ij})\\
&\quad+\mcl L_X\Si_{ij}-\frac1ng_{ij}-\frac{2N\tau}{n}\Si_{ij}-\na_i\na_jN\\
\De N&=-1 + N\Big[\abg{\Si}^2+\frac{\tau^2}{n}-n\Big]\\
\Delta X^i+R^i_{\,m}X^m-\mcl L_XV^i&=2\nabla_jN\Si^{ji}+\tau(2/n-1)\nabla^iN\\
&\quad-(2N\Si^{mn}- (\mcl L_Xg)^{mn})(\Chr imn-\hat{\Gamma}^i_{mn})
}
}
where the second fundamental form $k$ has been decomposed into
\eq{
k=\Si+\frac{\tau}n g,
}
where $\Si$ denotes the tracefree part. Note it is assumed that $N>0$. We have used the following standard notations. $R(g)$ denotes the Ricci scalar curvature of $g$, $\mathcal L_X$ denotes the Lie derivative w.r.t.~the shift, $R_{ij}$ denotes the Ricci tensor of the metric $g$. The Laplacian $\Delta$ is understood to be defined w.r.t.~$g$.\\
Finally, we remark that in the case of an reversed CMC-gauge, $t=-\tau$, which we use for spatial Einstein metrics of positive curvature, one has the lapse equation in the form
\eq{\label{inverseCMC}
\De N=1 + N\Big[\abg{\Si}^2+\frac{\tau^2}{n}-n\Big].
}
The equation for the trace free part of the second fundamental form in this case reads
\eq{\alg{
\p t\Si_{ij}&=N(R_{ij}+\tau \Si_{ij}-2\Si_{il}\Si_j^l+(\tau^2/n-n)g_{ij})\\
&\quad+\mcl L_X\Si_{ij}+\frac1ng_{ij}-\frac{2N\tau}{n}\Si_{ij}-\na_i\na_jN
}}
and the other equations remain the same.
\subsection{The Einstein operator}
The fundamental property of the CMCSH-Einstein flow lies in its elliptic-hyperbolic structure given by the decomposition of the spatial Ricci tensor into the Einstein-operator and perturbation terms as given in the following lemma.
\begin{lem}[{\cite[Lemma 6.2]{AMo11}}]\label{lem-ricc-dec}
Let $\ga$ be an Einstein metric, then we have the expansion
\eq{
R_{ij}-\delta_{ij}-\frac{R(\ga)}{n}g_{ij}=\frac12 \mathcal L_{g,\ga}(g-\ga)_{ij}+J_{ij},
}
where 
\eq{
\mathcal{L}_{g,\ga}h=-\Delta_{g,\ga} h-2\overset{\,\circ}{R}_{\ga}h
}
is the Einstein operator. The Laplacian has the local formula
\eq{\Delta_{g,\gamma}h_{ij}=\frac{1}{\mu_g}\nabla[\gamma]_m(g^{mn}\mu_g\nabla[\gamma]_n h_{ij}),}
the curvature action is given by $\overset{\,\circ}R_\gamma(h)_{ij}=R_{ikjl} (\gamma)h^{kl}$ for some 2-tensor $h$, $\delta_{ij}\equiv\frac12(\na_iV_j+\na_jV_i)$ and $J$ is an error term which satisfies the estimate
\eq{\label{est-J}
\Abk{J}{s-1}\leq C \Abk{g-\ga}s^2.
}
\end{lem}

\subsection{Homogeneous solutions of the CMC-Einstein flow}\label{ssec : 24}
We recover the background solutions discussed before in the present gauge by assuming homogeneity. In the homogeneous setting, meaning a vanishing trace free part of the second fundamental form, a spatially constant lapse function and vanishing shift vector we obtain the following solutions.
\subsubsection*{Standard CMC-gauge}
For the standard CMC-gauge, $t=\tau$,  we deduce from the lapse equation that
\eq{
N=\frac{n}{\tau^2-n^2}
}
and since $N>0$, $\tau^2>n^2$.
Then the evolving physical metric is given by
\eq{
g(\tau)=g(\tau_0)\frac{\tau_0^2-n^2}{\tau^2-n^2},
}
where the metrics have the property that the Ricci tensor is given by $R_{ij}=-\frac{n-1}{nN}g_{ij}$. We recover the metric \eqref{model1} written in CMC time.
In particular, the scalar curvature of the Einstein manifold is negative.\\
\subsubsection*{Reversed CMC-gauge}

\noindent Suppose now, we have the reversed CMC-gauge, i.e.\ $t=-\tau$. 
Then
\eq{
N=\frac{n}{n^2-\tau^2}
}
and therefore, $\tau^2<n^2$. The solution for curve of the physical metric is the 
same as above, where
$R_{ij}=\frac{n-1}{nN}g_{ij}$, i.e.\ the scalar curvature is positive. We recover \eqref{model2}. 




\subsection{Universality of the gauge conditions}\label{ssec : 26}

An important issue arising in the context of the CMCSH gauge concerns the generality of perturbations which can be evolved by the CMCSH Einstein flow. Considering CMC initial data induced by a background solution which admits a CMC foliation one would prefer to consider general perturbations of this initial data, i.e.~solutions to the general constraint equations without the CMC condition. But such non-CMC initial data cannot be evolved by the CMCSH Einstein flow.\\
This problem can be overcome by the following construction. Assuming non-CMC initial data close to the initial data induced by the background solution, which is CMC, general theory assures the existence of a maximal globally hyperbolic development
of this initial data. Under the smallness condition it is possible to show that this development contains a CMC surface. Starting the evolution from this CMC surface one can analyze the geometry of the corresponding future development and eventually treat all perturbations by considering the equations in CMCSH gauge.\\
A similar question arises for the spatial harmonic gauge, which also is shown to apply to general initial data. We discuss these aspects of the gauges in the following. 
\subsubsection{Spatial Harmonicity}
Let $\mathcal{M}$ be the set of smooth Riemannian metrics on $M$.
Fix a metric $\gamma\in \mathcal{M}$ and let $\mathcal{H}$ be the set of metrics $g\in \mathcal M$ such that $\identity:(M,g)\to(M,\gamma)$ is a harmonic map. In other words,
\begin{align}\label{SHmetrics}
\mathcal{H}=\left\{g\in\mathcal{M}\mid V_g^k=g^{ij}(\Gamma[g]_{ij}^k-\Gamma[\gamma]_{ij}^k)=0\right\},
\end{align}
where $\Gamma[g],\Gamma[\gamma]$ are the Christoffel symbols of $g,\gamma$, respectively.
Our aim in this section is to prove that under certain conditions on the background metric $\gamma$, $\mathcal{H}$ is a smooth submanifold of $\mathcal{M}$ and a local slice of the action of the diffeomorphism group through $\gamma$.
By the first variation of the Christoffel symbols (see e.g.\ \cite[Theorem 1.174]{Bes08}), the differential of the map $\Phi:g\mapsto V_g$ is given by
\begin{align}
d\Phi(h)^i=\nabla^jh_j^{\brk i}-\frac{1}{2}\nabla^i\trace_g h, 
\end{align}
where $h$ is a symmetric 2-tensor and $\na$ is the covariant derivative w.r.t.~$g$.
\begin{lem}\label{Piso}Let $(M,\gamma)$ be an Einstein manifold such that
$-2/n\cdot R(\ga)$ is not an eigenvalue of the Laplacian $\Delta_\ga$ and $\gamma$ does not admit Killing vector fields. Then the operator

\begin{align}
P:X^i\mapsto \Delta_\ga X^i+R_j^{i}[\ga]X^j
\end{align}
is an isomorphism which preserves the decomposition
\begin{align}\mathfrak{X}(M)=\left\{\gradient f\mid f\in C^{\infty}(M)\right\}\oplus \left\{ X\in C^{\infty}(TM)\mid \diver X=0\right\}.
\end{align}
\end{lem}
\begin{proof} We suppress the dependance on $\ga$ in the following notation.
By a standard argument using commutators of covariant derivatives, we have
\begin{align}
\Delta \nabla^if+R_j^{i}\nabla^jf=\nabla^i\Delta f+2R_j^{i}\nabla^if=\nabla^i\Delta f+\frac{2 R}{n}\cdot \nabla^if
\end{align}
which shows that because of the eigenvalue assumption, $P$ maps the first factor bijectively onto itself. By self-adjointness of $P$, the second factor is also preserved. We define maps $L$ and $L^*$ by 
\begin{align}
L: X\mapsto\frac{1}{2}(\nabla_iX_j+\nabla_jX_i),\qquad L^*:h\mapsto-\nabla^jh_j^{\brk i}.
\end{align}
Note that $L^*$ is the adjoint map of $L$ with respect to the $L^2$-scalar product induced by $\ga$. Now for any vector field $X$ with $\diver X=0$, we have
\begin{align}(L^*L X)^i=
\nabla^j\nabla_jX^i+\nabla^j\nabla^iX_j
 =\Delta X^i+\nabla^j\nabla^iX_j-\nabla^i\nabla^jX_j
 =\Delta X^i+R_j^{i}X^j=(PX)^i.
\end{align}
Thus, $PX=0$ implies $LX=0$. But the kernel of $L$ contains precisely the Killing vector fields, and hence, $X=0$. Therefore, $P$ is injective and by self-adjointness, $P$ is also surjective. 
\end{proof}
\begin{lem}\label{splittinglemma}
Let $(M,\gamma)$ be an Einstein manifold such that
$-2/n\cdot R(\ga)$ is not an eigenvalue of the Laplacian $\Delta_\ga$ and $\gamma$ does not admit Killing vector fields.
Then, $d\Phi_{\gamma}:C^{\infty}(S^2M)\to C^{\infty}(TM)$ is surjective.
Moreover, we have the splitting
\begin{align}\label{splitting}
C^{\infty}(S^2M)=\kernel( d\Phi_{\gamma})\oplus \image(L).
\end{align}
Here, $S^2M$ denotes the bundle of symmetric $2$-tensors.
\end{lem}
\begin{proof}
To prove the first assertion, we consider a vector field $X$ and compute
\begin{align}\label{bochner}
 d\Phi_{\gamma}\circ L(X)^i=\frac{1}{2}(\nabla^j\nabla_jX^i+\nabla^j
 \nabla^iX_j)-\frac{1}{2}\nabla^i\nabla_j X^j=\Delta X^i+R^i_j X^j=(PX)^i.
 \end{align}
Due to Lemma \ref{Piso}, $P$ is an isomorphism, so $d\Phi_{\gamma}$ is surjective even when restricted to Lie derivatives. This calculation also shows that $\kernel( d\Phi_{\gamma})\cap \image(L)=0$.
To prove that the direct sum spans all of $C^{\infty}(S^2M)$, we first note that
\begin{align}
\kernel( d\Phi_{\gamma})\oplus \image(d\Phi_{\gamma}^*)=C^{\infty}(S^2M)
\end{align}
so it suffices to prove that
\begin{align}
 \image(d\Phi_{\gamma}^*)\subset \kernel( d\Phi_{\gamma})\oplus \image(L).
\end{align}
If $h\in\image(d\Phi_{\gamma}^*)$, there is a vector field $X$ such that
\begin{align}
h=d\Phi_{\gamma}^*(X)=-LX+\frac{1}{2}\diver(X)\cdot\gamma.
\end{align}
We may use an arbitrary vector $Y$ field to rewrite this expression as
\begin{align}
h=-L(X+Y)+LY+\frac{1}{2}\diver(X)\cdot\gamma
\end{align}
We are done with the proof if we can find $Y$ such that
\begin{align}
LY+\frac{1}{2}\diver(X)\cdot\gamma\in  \kernel( d\Phi_{\gamma}).
\end{align}
Thus, we have to solve
\begin{align}
0= d\Phi_{\gamma}(LY+\frac{1}{2}\diver(X)\cdot\gamma)=PY+\frac{1}{2}d\Phi_{\gamma}(\diver(X)\cdot\gamma)
\end{align}
where we used \eqref{bochner}. This can be done due to Lemma \ref{Piso}. 
\end{proof}
\noindent For our purposes, it is more convenient to work on neighbourhoods with Sobolev regularity. We therefore use $H^s$-norms with $s>\frac{n}{2}+1$ for the following theorem. We remark that the above lemmas also hold, if we descend to $H^s$-regularity. Let $\mathcal{M}^s$ be the space of $H^s$-metrics on $M$ and let $\mathcal{H}^s$ be the set of all $g\in \mathcal{M}^s$ satisfying the condition in \eqref{SHmetrics}. 
\begin{thm}\label{slicethm}
Let $(M,\gamma)$ be an Einstein manifold such that
$-2/n\cdot R(\ga)$ is not an eigenvalue of the Laplacian $\Delta_\ga$ and $\gamma$ does not admit Killing vector fields. Then in a small $H^s$-neighbourhood $\mathcal{U}\subset \mathcal{M}^s$ of $\gamma$, $\mathcal{H}^s$ is a smooth submanifold of $\mathcal{M}^s$ with tangent space
\begin{align}
T_{\gamma}\mathcal{H}^s=\left\{h\in H^s(S^2M)\mid
d\Phi(h)^i=\nabla^j h_j^i-\frac{1}{2}\nabla^i\trace h =0\right\}.
\end{align}
Moreover, for any $g\in\mathcal{U}$ there exists an isometric metric $\tilde{g}\in\mathcal{H}^s$ which is $H^s$-close to $\gamma$, i.e.\ there exists $\varphi\in H^s(\Diff(M))$ such that $g=\varphi^*\tilde{g}$.
\end{thm}
\begin{proof}
The first assertion follows from the first assertion of Lemma \ref{splittinglemma}.
 The second assertion follows from the implicit function theorem for Banach manifolds applied to the map
 \begin{align}\Psi:\mathcal{H}^s\times  H^s(\Diff(M))\to \mathcal{M}^s
 \end{align}
 given by $\Psi(g,\varphi)=\varphi^*g$.
Since there are no Killing fields, $d\Psi_{(\gamma,\identity)}$ is injective and its image is
\begin{align}
\image(d\Psi_{(\gamma,\identity)})=T_{\gamma}\mathcal{H}^s\oplus\left\{\mathcal{L}_X\gamma\mid X\in H^s(TM)\right\}=\kernel( d\Phi_{\gamma})\oplus\image(L)
\end{align} 
which equals $H^s(S^2M)=T_{\gamma}\mathcal{M}^s$ by \eqref{splitting}.
Therefore, $\Psi$ is a diffeomorphism from a $H^s$-neighbourhood of $(\gamma,\identity)$ in $\mathcal{H}^s\times  H^s(\Diff(M))$ to a $H^s$-neighbourhood of $\gamma$ in $\mathcal{M}^s$.
\end{proof}
\begin{rem}
The assertions of Theorem \ref{slicethm} hold for any Riemannian metric $\gamma$ where the operator $P$ is an isomorphism.
\end{rem}

\subsubsection{Constant mean curvature hypersurfaces}
Let us now consider the CMC-gauge. Let $M$ be a compact manifold, $I\subset\R$ an open and bounded interval and $\widetilde{\mathcal{M}}^{k,\alpha}$ the set of $C^{k,\alpha}$-Lorentz metrics on $I\times M$, such that the induced metrics on the hypersurfaces $M_t=\left\{t\right\}\times M$ are all Riemannian. A Banach manifold structure on this set is induced by the norm
\begin{align}
\left\|\tilde{g}\right\|_{C^{k,\alpha}}=\left\|N^2\right\|_{C^{k,\alpha}}+\left\|X\right\|_{C^{k,\alpha}}+\left\|g_t\right\|_{C^{k,\alpha}},
\end{align}
where we identify $\tilde{g}$ according to the foliation by submanifolds $M_t$ with the triple $(N,X,g_t)$ of the Lapse function, the shift vector and the induced metrics $g_t=\tilde{g}|_{M_t}$. The norms of the right hand side are taken with respect to the Riemannian metric $dt^2+\gamma$ on $I\times M$. 
Let $C^{k,\alpha}(M,I)$ the set of $C^{k,\alpha}$-functions $f:M\to I$ endowed with the natural Banach manifold structure. Each such function defines naturally an embedding $\imath_f:M\to I\times M$ by $\imath_f(x)=(f(x),x)$. We define a map
\begin{align}
H:\widetilde{\mathcal{M}}^{k,\alpha}\times C^{k+1,\alpha}(M,I)\supset\mathcal{D}\to C^{k-1,\alpha}(M)
\end{align}
which associates to each pair $(\tilde{g},f)$ the mean curvature along the embedding $\imath_f:M\to I\times M$ induced by the metric $\tilde{g}$. 
Here, $\mathcal{D}$ is the open subset of pairs $(\tilde{g},f)$ such that $(\imath_f)^*\tilde{g}$  is a Riemannian metric.

This is a smooth map between Banach manifolds. To see this, it suffices to consider a local expression of this map:
Using local coordinates on $M$, we see that $T_{(f(x),x)}\image (\imath_f)$ is spanned by the vectors
$d \imath_f(\partial_i)=\partial_i f\cdot\partial_t+\partial_i$. Let $F(t,x)=t-f(x)$ and the matrix $(g_{f})_{ ij}$ be defined by $(g_f)_{ij}=\tilde{g}(d \imath_f(\partial_i),d \imath_f(\partial_j))$.
Then the mean curvature is
\eq{\alg{
H_{\tilde{g},f}=(g_f)^{ij}\tilde{g}(\tilde{\nabla}_{d \imath_f(\partial_i)}\nu,d \imath_f(\partial_j)),\qquad \nu=\frac{\gradient_{\tilde{g}} F}{|\gradient_{\tilde{g}} F|_{\tilde{g}}}
}}
where $\nu$ is the timelike unit normal and $(g_f)^{ij}$ is the inverse of $(g_f)_{ij}$. This expression contains second derivatives of the function $f$ and first derivatives of $\tilde{g}$.
We use an implicit function theorem applied to the map $H$ to prove the following lemma.
\begin{lem}
Let $(M,\gamma)$ be a compact Einstein manifold with scalar curvature $R(\ga)=-n(n-1)$
and let $I$ be an arbitrary open and bounded interval around $0$.
Let $\ell\geq 1$ and consider the metric
\begin{align*}
\tilde{\gamma}=-dt^2+\cosh^2(t)\gamma
\end{align*}
whose initial data induced on the hypersurface $\left\{0\right\}\times M$ is $(\gamma,0)$.
Then for any $C^{\ell,\alpha}\times C^{\ell-1,\alpha}$-neighbourhood $\mathcal{U}$ of $(\gamma,0)$, there exists a neighbourhood $\mathcal{V}\subset \widetilde{\mathcal{M}}^{\ell,\alpha}$ of $\tilde{\gamma}$
such that any $\tilde{g}\in\mathcal{V}$ admits a hypersurface such that the pair $(g,k)$ of the metric and the second fundamental form induced on this hypersurface is in $\mathcal{U}$ and $\trace_gk\equiv0$.
\end{lem}
\begin{proof}
Consider the map $H$ of above and note that $H(\tilde{\gamma},0)=0$.
We compute its differential  at the tupel $(\tilde{\gamma},0)$ restricted to the second argument. By the variational formula of the mean curvature in
\cite[Proposition 2.2]{BBC08},
\begin{align}
dH_{\tilde{\gamma},0}(0,w)=\frac{1}{n}[\Delta_\ga w-(\ric_{\tilde{\gamma}}(\partial_t,\partial_t)+|k|^2_{\gamma})w]
=\frac{1}{n}[\Delta_\ga w+n w].
\end{align}
Because we excluded the case of the sphere, the operator $\Delta_\ga+n:C^{\ell+1,\alpha}(M)\to C^{\ell-1,\alpha}(M)$ is an isomorphism \cite[Theorem 1 and Theorem 2]{Ob62}. Due to the implicit function theorem for Banach manifolds, we have neighbourhoods $\mathcal{U}'\subset C^{\ell,\alpha}(\widetilde{\mathcal{M}})$ of $\tilde{\gamma}$, $\mathcal{V}'\subset C^{\ell+1,\alpha}(M,I)$ of $0$ and a smooth function $F:\mathcal{U}'\to \mathcal{V}'$ such that $H(\tilde{g}, F(\tilde{g}))=0$, i.e. $F$ associates to each metric $\tilde{g}$ a minimal Riemannian hypersurface given by the graph of the function $ F(\tilde{g})$. Moreover, 
$F(\tilde{g})$ is the only function in $\mathcal{V}'$ such that $H(\tilde{g}, F(\tilde{g}))=0$. The proof is finished by the remark that the map $\tilde{g}\to (g,k)$ associating to $\tilde{g}$ the metric and the second fundamental form of $\mathrm{graph}(F(\tilde{g}))$ is continuous from $\mathcal{D}$ to $C^{\ell,\alpha}\times C^{\ell-1,\alpha}$.
\end{proof}
\begin{lem}
Let $(M,\gamma)$ be an Einstein manifold with scalar curvature $R(\ga)=-n(n-1)$
and let $I$ be an arbitrary open and bounded interval in $(0,\infty)$ around $t_0=\arcsinh(1)$.
Let $\ell\geq 1$ and consider the  metric
\begin{align*}
\tilde{\gamma}=-dt^2+\sinh^2(t)\gamma
\end{align*}
whose initial data induced on the hypersurface $\left\{\arcsinh(1)\right\}\times M$ is $(\gamma,-\sqrt{2}\gamma)$.
Then for any $C^{\ell,\alpha}\times C^{\ell-1,\alpha}$-neighbourhood $\mathcal{U}$ of $(\gamma,-\sqrt{2}\gamma)$, there exists a neighbourhood $\mathcal{V}\subset C^{\ell,\alpha}(\widetilde{\mathcal{M}})$ of $\tilde{\gamma}$
such that any $\tilde{g}\in\mathcal{V}$ admits a hypersurface such that the pair $(g,k)$ of the metric and the second fundamental form induced on this hypersurface is in $\mathcal{U}$ and $\trace_gk\equiv-\sqrt{2}n$.
\end{lem}
\begin{proof}
The proof is analogous as above. In this case, we consider the map
$\bar{H}=H+\sqrt{2}n$. We have $\bar{H}(\tilde{\gamma},t_0)=0$ and we compute
\begin{align}
d\bar{H}_{\tilde{\gamma},t_0}(0,w)=\frac{1}{n}[\Delta_\ga w-(\ric_{\tilde{\gamma}}(\partial_t,\partial_t)+|k|^2_{\gamma})w]
=\frac{1}{n}[\Delta_\ga w -n w].
\end{align}
which is always an isomorphism from $C^{\ell+1,\alpha}(M)$ to $C^{\ell-1,\alpha}(M)$.
\end{proof}
\begin{thm}\label{CMCthm}
Let $\gamma$ be an Einstein manifold of scalar curvature $R(\ga)=n(n-1)$ (resp. $R(\ga)=-n(n-1)$) and let $s>n/2+1$, $s'>n/2+s$. Then for any $H^s\times H^{s-1}$-neighbourhood $\mathcal{U}\ni(\gamma,0)$ (resp. $\mathcal{U}\ni(\gamma,-\sqrt{2}\gamma)$) of CMC initial data sets, there exists a $H^{s'}\times H^{s'-1}$-neighbourhood $\mathcal{V}\ni(\gamma,0)$ (resp. $\mathcal{U}\ni(\gamma,-\sqrt{2}\gamma)$) of general initial data sets such that any development of initial data in $\mathcal{V}$ admits a CMC-hypersurface such that the initial data induced on the hypersurface lies in $\mathcal{U}$.
\end{thm}
\begin{proof}
Let $(g_i,k_i)$ be an initial data set converging to $(\gamma,0)$ in $H^{s'}\times H^{s'-1}$.
 By the proof of \cite[Theorem 15.10]{Rin09}, one obtains a sequence of solutions of Einsteins equations $\tilde{g}_i$ such that for each slice $\left\{t\right\}\times M$, $t\in I$, the data $(N_i,X_i,(g_i)_t)$ converges in $H^{s'}$ to the corresponding data of the background solution $\tilde{\gamma}$. Moreover, we have $H^{s'-1}$-convergence of the time derivatives $(\partial_t N_i,\partial_t X_i)$ by the choice of the gauge used in the proof of the above mentioned theorem. Moreover, we have $H^{s'-1}$-convergence
of $\partial_t g_i$ to the corresponding quantity of $\tilde{\gamma}$. 

By Sobolev embedding, we obtain convergence of $( N_i, X_i,(g_i)_t)$ in $C^{s,\alpha}$ and convergence of their time-derivatives $(\partial_t N_i,\partial_t X_i,\partial_t g_i)$ in  $C^{s-1,\alpha}$ on each slice $\left\{t\right\}\times M$. Using the gauge condition and the Einstein equation, we also obtain convergence of higher time-derivatives of the above quantities so that $\tilde{g}_i$ converges to $\tilde{\gamma}$ in $C^{s,\alpha}(I\times M)$. For $i$ large enough, the metrics $\tilde{g}_i$ admit hypersurfaces of constant mean curvature due to the lemmas above and the initial datas $(\bar{g}_i,\bar{k}_i)$ induced on the hypersurfaces converge in $C^{s,\alpha}\times C^{s-1,\alpha}$, hence also in $H^s\times H^{s-1}$. This proves the theorem.
\end{proof}

\section{One-parameter family of initial data with collapsing and expanding regimes}
\label{sec : 4}

This section is concerned with the construction of a one-parameter family of initial data such that for a parameter value strictly above a certain threshold, the future and past development recollapses while for the critical value and below the corresponding future and past development expands for all time. The initial data consists of a product of positive Einstein metrics with identical Einstein constants while both metrics are multiplied by a large respectively by a small constant  - yielding non-equilibrium initial data, where the small factor recollapses. For parameter values close to 1 the Einstein metrics are initially of almost similar volume and expand both for infinite time. The initial data corresponding to the threshold value of the parameter yields a solution where one factor remains constant in time while the second metric expands for infinite time. We proceed with the explicit construction. \\

We consider a product manifold $M\times N$ such that $g_M$ and $g_N$ be Einstein metrics of positive scalar curvature on $M$ and $N$, respectively, with $\dim M=\dim N=m$. Let $n=2m$. The Einstein constants are chosen such that
\eq{\label{einst-pos}
\mathrm{Ric}_{g_N}=(n-1)g_N\quad\mbox{ and }\quad \mathrm{Ric}_{g_M}=(n-1)g_M.
}
Given $s\in(\tfrac12,\infty)$, let
\eq{\label{parametermetrics}
g_{M}(s)= s\cdot g_M\quad\mbox{ and }\quad g_{N}(s)=\frac{s}{2s-1}\cdot g_N.
}
We consider now a Lorentzian metric 
\eq{
\tilde{g}=-dt^2+a(t)^2 g_M(s)+b(t)^2g_N(s)
}
on $I\times M\times N$, where $I\subset\R$ is some interval.
$\tilde g$ is supposed to be a solution of the Einstein equations
\eq{
\mathrm{Ric}_{\tilde g}=n\tilde{g}
}
with initial conditions
\eq{\label{hom-in-cond}
a(0)=b(0)=1 \quad\mbox{ and }\quad  a'(0)=b'(0)=0,
}
which are compatible with the constraints. Furthermore, we define the new variables, $x=\log a$ and $y=\log b$. Einstein equations then imply the system of ODE's
\eq{\label{eq : wp-ev}\alg{
x''&=n-\frac1s(n-1)e^{-2x}-\frac n2[(x')^2+x'y']\\
y''&=n-(2-\frac1s)(n-1)e^{-2y}-\frac n2[y'^2+x'y']
}}
with initial data $x(0)=y(0)=x'(0)=y'(0)=0$. The equation $\tilde{R}_{00}=n\tilde g_{00}$ yields
\eq{
\frac n2(\frac{a''}{a}+\frac{b''}b)=n
}
and equivalently
\eq{\label{eq : wp-nb}
x''+y''+x'^2+y'^2=2.
}
If $s=1$, we recover the generalized de-Sitter metric since $a(t)=b(t)=\cosh(t)$ in this case.
\begin{thm}\label{1parameterthm}
Let $s\in(\tfrac12,\infty)$. Consider initial data $(g_M(s)\oplus g_N(s),0)$ on $M\times N$ where $g_M(s), g_N(s)$ are as in \eqref{einst-pos} and \eqref{parametermetrics}. Then, for $\frac{n-1}{n-2}<s$ or $s<\frac{n-1}{n}$, the future and past development is geodesically incomplete. For $s\in[\frac{n-1}{n} ,\frac{n-1}{n-2}]$ the future and past development is geodesically complete. Moreover, if $s\in (\frac{n-1}{n} ,\frac{n-1}{n-2})$, we have a limit
\eq{\label{limit a/b} C_s=\lim_{t\to\pm\infty}\frac{\volume(M,a(t)^2 g_M(s))}{\volume(N,b(t)^2 g_N(s))}.
}
\end{thm}
\begin{proof}Let us prove the first assertion. Without loss of generality, we restrict to the case $s>\frac{n-1}{n-2}$.
By Lemma \ref{non-equilibrium} below, there exists a time $T_0>0$ such that $\lim_{t\to T_+}y(t)=-\infty$ which in turn implies that the scale factor of the metric $g_N(s)$ satisfies $\lim_{t\to T_+}b(t)=0$. By choice of the initial values, $a$ and $b$ are time-symmetric, hence $\lim_{t\to -T_+}b(t)=0$ as well. 
 Clearly, these solutions are geodesically incomplete in the future and the past.
 
 Let us now prove the second assertion. The case $s=1$ is the case of the de-Sitter space, so there is nothing to prove. We may restrict to the case $1<s\leq \frac{n-1}{n-2}$.
 Then by Lemma \ref{equilibrium}, the functions $x,y$ hence the scale factors $a,b$ exist for all $t>0$. By time-symmetry, they exist for all $t\in \R$. All these solutions are future- and past geodesically complete. Finally, the existence of \eqref{limit a/b} follows from \eqref{limit x-y}.

\end{proof}

\subsection{Evolution of non-equilibrium initial data}

We consider first the case of non-equilibrium initial data, which we define by
\eq{
s>\frac{n-1}{n-2}.
}
If $s<\frac{n-1}{n}$ the roles of $g_M(s)$ and $g_N(s)$ interchange, so we may restrict to the first case.
We prove the following lemma.

\begin{lem}\label{non-equilibrium}
For non-equilibrium initial data initial data, the solution $(x(t),y(t))$ of the system
\eqref{eq : wp-ev} with initial data $x(0)=y(0)=x'(0)=y'(0)=0$ does not exist for all time.
More precisely, there exists a time $T_+>0$ such that $\lim_{t\to T_+}y(t)=-\infty$.
\end{lem}
\begin{proof}
The condition on $s$ implies
\eq{
y''(0)=n-(2-\frac1s)(n-1)<0
}
and
\eq{
x''(0)=n-\frac1s(n-1)>2.
}

We show that $x$ and $y$ are strictly monotonically increasing (decreasing, respectively) on the interval of existence.

Since $x''(0)>0$, we have $x'>0$ for small $t>0$. Let $t_0>0$ denote the first time, such that $x'(t_0)=0$. Then $\eqref{eq : wp-ev}$ implies
\eq{
x''(t_0)=n-\frac{1}{s}(n-1)e^{-2x(t_0)}>2>0,
}
which in turn implies $x'(t)<0$ for $t\in(t_1-\varepsilon,t_1)$. Therefore $x'>0$ as long as it exists. Analogously one can show that $y$ is strictly monotonically decreasing on the interval of existence.
Next we show, that $x$ exists at least as long as $y''$ und $y'$. By \eqref{eq : wp-nb} and the monotonicity of $x$ and $y$ we obtain
\eq{\alg{
x''-y''&=(n-1)[(2-\frac{1}{s})e^{-2y}-\frac{1}{s}e^{-2x})]-\frac{n}{2}[(x')^2-(y')^2]\\
                     &\leq2(n-1)-\frac{n}{2}[(x')^2-(y')^2]
}}
Using \eqref{eq : wp-nb}, $x''$ can be eliminated from this inequality. Elementary manipulations yield

\eq{
\label{bounded}(x')^2\leq 4+\frac{4}{n-1}y''+\frac{n}{n-2}(y')^2.
}
which proves the claim.
We show in the following that $y$ is unbounded from below. Assume the contrary. Then by the strict monotonicity of $y$ the existence of a limit $\lim_{t\to\infty}y(t)=:y(\infty)$ follows, so does $\lim_{t\to\infty}y'(t)=0$. Then
\eq{
\int_0^{\infty}y''(t)dt=\lim_{t\to\infty}y'(t)-y'(0)=0,
} so either $\lim_{t\to\infty}y''(t)=0$ or there is a sequence $s_i\to \infty$ such that $y''(s_i)=0$ for all $i\in\N$. Using \eqref{eq : wp-ev} would then imply $x'(t)$ diverges for $t\to\infty$, which however is a contradiction to \eqref{bounded}.\\
\noindent
Finally, we show that $y$ blows up in finite time. Addition of both equations in \eqref{eq : wp-ev} yields
\eq{\label{x+y}
x''+y''=2n-\frac{1}{s}(n-1)e^{-2x}-(2-\frac{1}{s})(n-1)e^{-2y}-\frac{n}{2}(x'+y')^2.
}
By monotonicity of $x$ and $y$ and the unboundedness of $y$ from below there is a constant $C>0$ and a time $t_1>0$ such that the differential inequality 
\eq{
x''+y''\leq-C-\frac{n}{2}(x'+y')^2
}
holds for all $t\geq t_1$. Here $t_1>0$ is arbitrary. The corresponding ODE is solved by the tangens. Therefore $x'+y'$ blows up in finite time (say $t_2>t_1$) towards $-\infty$. As $x'>0$, $y'$ blows up. In addition, we also have $y(t_2)=\int_0^{t_2}y'(t)dt=-\infty$, which implies that $y$ diverges.
\end{proof}

\subsection{Equilibrium initial data}\label{kleine_s}
We consider now equilibrium initial data given by
\eq{
1<s\leq \frac{n-1}{n-2}.
}
\begin{lem}\label{equilibrium}
For equilibrium initial data initial data,
the solution $(x(t),y(t))$ of the system
\eqref{eq : wp-ev} with initial data $x(0)=y(0)=x'(0)=y'(0)=0$ exists for all $t>0$. Moreover, if $s< \frac{n-1}{n-2}$, we have a limit
\eq{\label{limit x-y}C_s=\lim_{t\to\infty}(x(t)-y(t)).}
\end{lem}
\begin{proof}
We consider the case of $1<s<\tfrac{n-1}{n-2}$. Then
\eq{
y''(0)=n-(2-\frac{1}{s})(n-1)\in(0,1)
}
and
\eq{
x''(0)=n-\frac{1}{2}(n-1)\in (1,2).
}
Both $x$ and $y$ are strictly monotonically increasing. We have $x'(t)>0$ for small $t$. Let $t_0$ be the first time, where $x'(t_0)=0$. Then by $\eqref{eq : wp-nb}$ we obtain
\eq{
x''(t_0)=n-\frac{1}{s}(n-1)e^{-2x(t_0)}>0,
}
as long as $y'(t_0)$ exists. Thus $x'<0$ on $(t_0-\varepsilon,t_0)$, which causes the contradiction.
 Analogously one shows that $y$ is strictly monotonically increasing. From \eqref{x+y} and monotonicity we deduce
\eq{
2-\frac{n}{2}(x'+y')^2\leq x''+y''\leq 2n-\frac{n}{2}(x'+y')^2.
}
The solution of the corresponding ODE is $\tanh$. This implies $0<x'+y'<C$ for all $t>0$. Due to the positivity of $x'$ and $y'$ these statements hold for $x'$ and $y'$ individually.
In particular, $x$ and $y$ exist for all times.
In addition we have $x'(t)+y'(t)>C_1>0$ for all $t\geq t_1$ and $x(t)+y(t)>C_2t$ for all $t\geq t_1$ and $C_2>0$. Using what we have seen so far we obtain the following estimates.
\eq{\alg{
x''-y''&=(n-1)[(2-\frac{1}{s})e^{-2y}-\frac{1}{s}e^{-2x})]-\frac{n}{2}[(x')^2-(y')^2]\\
                     &\leq(n-1)(2-\frac{1}{s})e^{-2(x+y)}-\frac{n}{2}(x'-y')(x'+y')\\
                     &\leq(n-1)(2-\frac{1}{s})e^{-2C_2 t}-C_3(x'-y')\\
}}
This differential inequality holds for $t>t_1$. On the other hand, we also have
\eq{
 x''-y''\geq-(n-1)\frac{1}{s}e^{-2C_2 t}-C_4(x'-y').
}
From these inequalities we deduce that $x'-y'$ decays exponentially and converges to $0$ as $t\to\infty$. The exponential decay implies the existence of the limit 
\eq{
\lim_{t\to\infty}(x(t)-y(t))=\int_0^{\infty}(x'(t)-y'(t))dt.
}
In the boundary case $s=\tfrac{n-1}{n-2}$, $y''(0)=0$. Thus, $y\equiv 0$ and the system reduces to the initial value problem
\eq{x''=n-(n-2)e^{-2x}-\frac{n}{2}(x')^2,\qquad x(0)=x'(0)=0.}
By similar arguments as above, one shows that $x$ is strictly monotonically increasing. An immediate implication is $x''<n$, which implies that $x$ grows at most quadratically. Therefore, it it exists for all time.
\end{proof}


\subsection{Products of negative curvature}
Finally, we address the case of products of negative Einstein metrics. We consider an analogous construction as in the previous sections.
Let $g_M$ and $g_N$ be two compact $m$-dimensional Einstein metrics with 
\eq{\label{einst-neg}\alg{
\ric_{g_M}&=-(n-1)g_M\qquad\text{ and }\qquad\ric_{g_N}&=-(n-1)g_N,
}}
where $n=2m$. Let $s\in(\frac{1}{2},\infty)$ and
\eq{\label{parametermetrics2}
g_{M}(s)=s\cdot g_M\qquad g_{N}(s)=\frac{s}{2s-1}\cdot g_N
}
We consider a Lorentzian metric of the form
\eq{
\tilde{g}=-dt^2+a(t)^2g_{M}(s)+b(t)^2g_{N}(s),
}
and demand $\ric_{\tilde{g}}=n\cdot \tilde{g}$. We have the conditions $a(0)=b(0)=1$ and $a'(0)=b'(0)=\sqrt{2}$ compatible with the constraints. Defining the variables $x=\log(a)$ und $y=\log(b)$ yields the system of ODE's
\eq{\alg{x''&=n+\frac{1}{s}(n-1)e^{-2x}-\frac{n}{2}[(x')^2+x'y']\\
   y''&=n+(2-\frac{1}{s})(n-1)e^{-2y}-\frac{n}{2}[(y')^2+x'y']
}}
with initial conditions $x(0)=y(0)=0$, $x'(0)=y'(0)=\sqrt{2}$. Due to $\tilde{R}_{00}=n\cdot \tilde{g}_{00}$, we have
\eq{
\frac{n}{2}(\frac{a''}{a}+\frac{b''}{b})=n
}
or equivalently
\eq{
\label{ZB}x''+y''+(x')^2+(y')^2=2.
}
In the case $s=1$, we recover the background metric
\eq{ -dt^2+\sinh^2(t-\arcsinh(1))(g_M\oplus g_N).
}
For this system we obtain the
following result.
\begin{thm}
Let $s\in(\tfrac12,\infty)$. Consider initial data $(g_M(s)\oplus g_N(s),-\sqrt2(g_M(s)\oplus g_N(s)))$ on $M\times N$ where $g_M(s), g_N(s)$ are as in \eqref{einst-neg} and \eqref{parametermetrics2}. Then, the future development is geodesically complete. Moreover, we have a limit
\eq{ C_s=\lim_{t\to\infty}\frac{\volume(M,a(t)^2 g_M(s))}{\volume(N,b(t)^2 g_N(s))}.
}
\end{thm}
\begin{proof}
For any $s\in(\tfrac12,\infty)$, an analogoue of Lemma \ref{equilibrium} can be proven by the same arguments. Then the theorem follows as in the second part of Theorem \ref{1parameterthm}.
\end{proof}




\section{Nonlinear Stability}\label{sec : 6}
We turn now to the main part of the paper, presenting the proof of the nonlinear stability results, Theorem \ref{main-thm-neg} and Theorem \ref{main-thm}. The proof consists of four steps: a rescaling of the system, local stability results, elliptic estimates for lapse and shift and a uniform energy estimate for the evolving data $g$ and $\Si$. The steps distinguish formally between the cases of positive and negative curvature of the background geometry. However, both cases can be handled more or less similarly, so that we present most of the arguments only for one case explicitly and in detail.  
\subsection{CMC-Einstein flow in rescaled time}
Let us relabel the solution $(g,\Sigma,B,X)$ of the elliptic-hyperbolic system by
$\tilde{g},\tilde{\Sigma},\tilde{N},\tilde{X}$.
We now introduce new scale-invariant variables by
\eq{
g=s(\tau)\tilde{g},\qquad N=s(\tau)\tilde{N},\qquad \Sigma=s(\tau)^{1/2}\tilde{\Sigma},\qquad X=s(\tau)^{1/2}\tilde{X},
}
where $s(\tau)$ is defined below.
\subsubsection{Rescaling of the CMC flow}
In the CMC case, we define the scale factor as $s(\tau)=(\frac{\tau}{n})^2-1$.
In these variables, the constraint equations read
\eq{
R(g)-|\Sigma|_{g}^2=-(n-1)n,\qquad
\nabla^{i}\Sigma_{ij}=0.
}
Furthermore we define a new time-variable $\T$ by the equation $\tau=-n\frac{\cosh(\T)}{\sinh(\T)}$. This time coincides with the time of the solution \eqref{model1}.
We now rewrite this system in the rescaled variables and the time variable $\T$.
The defining equations for lapse and shift are
\eq{\label{resc-evol-1}\alg{
\Delta_{g}N&=-1+N(|\Sigma|_{g}^2+n)\\
\Delta_{g}X^i+R^i_mX^m&=2\nabla_jN\Sigma^{ij}-\cosh(\T)
(2-n)\nabla^iN-(2N\Sigma^{mn}-(\mathcal{L}_{X}g)^{mn})
(\Gamma_{mn}^i-\hat{\Gamma}_{mn}^i).
}}
Here we additionally used that $\nabla_j\Sigma^{ij}=0$ and $V^i=g^{ij}(\Gamma_{ij}^k-\hat{\Gamma}_{ij}^k)=0$
The evolution equations are
\eq{\label{resc-evol-2}\alg{
\partial_\T g_{ij}=&-2\frac{\cosh(\T)}{\sinh(\T)}(1-nN)g_{ij}-\frac{n}{\sinh(\T)}(2N\Sigma_{ij}-\mathcal{L}_Xg_{ij})\\
\partial_\T\Sigma_{ij}=&-n^2\frac{\cosh(\T)}{\sinh(\T)}(\frac{1}{n^2}+N-\frac{2N}{n})\Sigma_{ij}+\frac{n}{\sinh(\T)}N(R_{ij}+ng_{ij}-2\Sigma_{ik}\Sigma^k_j)\\
&+\frac{n}{\sinh(\T)}(\mathcal{L}_X\Sigma_{ij}-\frac{1}{n}g_{ij}-\nabla_{i}\nabla_{j}N).
}}

\subsubsection{Rescaling of the reversed CMC flow}
In the reversed CMC case, we rescale with $s(\tau)=1-(\frac{\tau}{n})^2$.
 Then,
some signs change. The constraints are
\eq{
R(g)-|\Sigma|_{g}^2=(n-1)n,\qquad
\nabla^{a}\Sigma_{ab}=0.
}
The defining equations for lapse and shift are
\eq{\alg{
\Delta_{g}N=&1+N(|\Sigma|_{g}^2-n),\\
\Delta_{g}X^i+R^i_mX^m=&2\nabla_jN\Sigma^{ij}-\sinh(\T)
(2-n)\nabla^iN
-(2N\Sigma^{mn}-(\mathcal{L}_{X}g)^{mn})
(\Gamma_{mn}^i-\hat{\Gamma}_{mn}^i)
}}
and the evolution equations are
\eq{\alg{
\partial_\T g_{ij}=&-2\frac{\sinh(\T)}{\cosh(\T)}(1-nN)g_{ij}-\frac{n}{\cosh(\T)}(2N\Sigma_{ij}-\mathcal{L}_Xg_{ij}),\\
\partial_\T\Sigma_{ij}=&-n^2\frac{\sinh(\T)}{\cosh(\T)}(\frac{1}{n^2}+N-\frac{2N}{n})\Sigma_{ij}+\frac{n}{\cosh(\T)}N(R_{ij}+ng_{ij}-2\Sigma_{ik}\Sigma^k_j)\\
&+\frac{n}{\cosh(\T)}(\mathcal{L}_X\Sigma_{ij}+\frac{1}{n}g_{ij}-\nabla_i\nabla_jN).
}}


\subsection{Local existence}\label{ssec : 25}
We have the following local existence theorem in CMCSH gauge, for the initial data, which we consider in this paper.
We distinguish between the cases of positive and negative curvature beginning with the latter.
\begin{lem}[Analogous to {\cite[Theorem 5.1]{AnMo03}}]
Let $\ga$ be a fixed Einstein metric on $M$ such that
$\ga$ is a metric of negative scalar curvature and $s>n/2+1$. Furthermore, let $(g_0,k_0)$ be CMCSH initial data on $M$ such that
\eq{
\Ab{g_0-\ga}_{H^s}+\Ab{\Si}_{H^{s-1}}<\varepsilon
}
with $\varepsilon$ sufficiently small. Then the CMCSH Cauchy problem is strongly locally well-posed in $\mathcal C^k(\mathcal H^s)$, $k=\lfloor s\rfloor$ and the corresponding Lorentz metric $\bar g$ is a vacuum solution of the Einstein equations. The following continuation principle holds.
There exists a $\delta>0$ such that for $[T_0,T_+)$ being the maximal future existence interval to the given initial data at $T_0$ in the rescaled time $T$, then either $[T_0,T_+)=[\arcsinh(1),\infty)$ or  
\eq{
\limsup (\Ab{g-\ga}_{H^s}+\Ab{\Si}_{H^{s-1}})\geq \varepsilon+\delta
}
for $\T\rightarrow \T_+$.

\end{lem}
The positive case is a bit more subtle due fact that the elliptic operators for lapse and shift and not necessarily isomorphisms. Recall the definition of $\mathcal B_\varepsilon^s(\ga,0)$ in section \ref{sec:nac}.
\begin{lem} Let $\ga$ be a fixed Einstein metric on $M$ such that
$\ga$ is a metric of positive scalar curvature, $-2(n-1)\notin\mathrm{Spec}(\Delta_{\ga})$ and $\ga$ admits no Killing vector fields and let $s>n/2+1$. Furthermore, let $(g_0,k_0)$ be CMCSH initial data on $M$ such that
\eq{
\Ab{g-\ga}_{H^s}+\Ab{\Si}_{H^{s-1}}<\varepsilon
}
with $\varepsilon$ sufficiently small to assure that the conditions on $\ga$ hold for all $(g,\Si)\in\mathcal B_\varepsilon^s(\ga,0)$. Then
the CMCSH Cauchy problem is strongly locally well-posed in $\mathcal C^k(\mathcal H^s)$, $k=\lfloor s\rfloor$ and the corresponding Lorentz metric $\bar g$ is a vacuum solution of the Einstein equations. The following continuation principle holds. There exists a $\delta>0$ such that for $(\T_-,\T_+)$ being the maximal existence interval to the given initial data at $\T_0$ in rescaled time $\T$, either $(\T_-,\T_+)=(-\infty,\infty)$ or 
\eq{
\limsup (\Ab{g-\ga}_{H^s}+\Ab{\Si}_{H^{s-1}})\geq \varepsilon+\delta
}
for $\T\rightarrow \T_+$ or $\T\rightarrow \T_-$.
\end{lem}
\begin{proof}
The lemma for the negative case follows straightforward by the same methods as in \cite{AnMo03}, which would even yield a more general result without the smallness assumptions.\\
In the case of positive curvature in the second lemma one has to assure that the elliptic operators defining the lapse and shift equation are in fact isomorphisms to use the relevant structure of the elliptic system as in lemma 3.2 of \cite{AnMo03}. The conditions we impose on the Einstein metric assure that these operators are isomorphisms, c.f.~section \ref{ssec : 26}. As the perturbations are chosen to be small we can assure that the isomorphism property holds also for the perturbed data as long as we remain in an $\varepsilon$-ball. This justifies the continuation criterion, which automatically covers the case implied by the analysis in \cite{AnMo03}.
\end{proof}

\begin{rem}
The local existence results mentioned above hold under more general conditions on the initial data (cf.~\cite{AnMo03}). We have stated a concise version which covers the case which is needed in the present paper. It is also understood that we choose the bootstrap assumptions in the proof of global existence such that the solution is contained inside the corresponding $\varepsilon$-ball of the corresponding local existence criterion.
\end{rem}


\subsection{Elliptic system} \label{ssec : ell}
We derive the elliptic estimates for lapse and shift in the case of positive and negative Ricci-curvature of the spatial metric.

\begin{lem}\label{ellipticestimates}
Let $s>n/2+2$ and $(g,\Sigma)\in \mathcal B^s_{\delta_g}(\gamma)\times \mathcal B_{\delta_\Sigma}^{s-1}(0)$ for some sufficiently small $\delta_g,\delta_\Sigma>0$, then

\eq{\alg{
\Ab {N-\frac{1}{n}}_{s}&\leq C(\delta_g,\delta_\Sigma)\Ab{\Sigma}_{s-2}^2,\\
\Ab{N}_{L^\infty}&\leq \frac{1}{n},\\
\Ab{X}_{s}&\leq C(\delta_g,\delta_\Sigma)\Big[\Ab{2\nabla_j{N}{\Sigma}^{ij}-2{N}{\Sigma}^{mn}(\Gamma_{mn}^i-\hat{\Gamma}_{mn}^i)}_{s-2}\\
&\qquad\qquad\qquad+n\cosh(\T)
(1-\frac{2}{n})\Ab{\nabla^i{N}}_{s-2}\Big] .
}}
In combination with the first inequality this implies for $X$,
\eq{\label{Xestimate}\alg{
\Ab{X}_{s}&\leq C(\delta_g,\delta_\Sigma)\Big[\Ab{{\Sigma}}^3_{s-2}+\Ab{N}_{s-2}\Ab{{\Sigma}}_{s-2}\Ab{g-\gamma}_{s-1}\\
&\qquad\qquad\qquad+n\cosh(\T)
(1-\frac{2}{n})\Ab{\Sigma}^2_{s-2}\Big] .
}}
In the case of positive curvature, $\cosh(\T)$ is replaced by $\sinh(\T)$ in the estimate. 
\end{lem}

\begin{proof} We prove the estimates in the case of negative Ricci curvature. The positive case is analogous.\\
The lapse equation, in the negative case, reads
\eq{
\Delta_{{g}}{N}=-1+{N}(|{\Sigma}|_{{g}}^2+n).
}
The maximum principle immediately yields the second estimate. Rewriting the lapse equation, we obtain
\eq{
\Delta_g(N-\frac{1}{n})-n(N-\frac{1}{n})=N\abg{\Sigma}^2,
}
which in combination with the point wise estimate on $N$ and elliptic regularity for the operator $\Delta_g-n$ yields the first estimate.\\

\noindent Finally, we consider the estimate for the shift vector. We write the equation for the shift in the form
\eq{\alg{
\Delta_{{g}}{X}^i+{R}^i_m{X}^m&=F_X+(\mathcal{L}_{{X}}{g})^{mn}
(\Gamma_{mn}^i-\hat{\Gamma}_{mn}^i),
}}
where
\eq{
F_X\equiv 2\nabla_j{N}{\Sigma}^{ij}-n\cosh(\T)
(1-\frac{2}{n})\nabla^i{N}-2{N}{\Sigma}^{mn}(\Gamma_{mn}^i-\hat{\Gamma}_{mn}^i).
}
Elliptic regularity applied to the equation for the shift then implies
\eq{
\Ab{X}_{H^{s}}\leq C(\delta_g)\Big[\Ab{F_X}_{H^{s-2}}+\Ab{(\mathcal{L}_{{X}}{g})^{mn}
(\Gamma_{mn}^i-\hat{\Gamma}_{mn}^i)}_{H^{s-2}}\Big].
}
Using the smallness of $g-\gamma$ we estimate
\eq{
\Ab{(\mathcal{L}_{{X}}{g})^{mn}
(\Gamma_{mn}^i-\hat{\Gamma}_{mn}^i)}_{H^{s-2}}\leq C(\delta_g)\Ab{X}_{H^{s-1}}\Ab{g-\gamma}_{H^{s-1}}.
}
By choosing the $\delta_g$ sufficiently small we can estimate the RHS by $\frac{1}{2}\Ab{X}_{H^{s-1}}$ and absorb it into the LHS above yielding
\eq{
\Ab{X}_{H^{s}}\leq 2C(\delta_g)\Ab{F_X}_{H^{s-2}}.
}
This finishes the proof.
\end{proof}

\subsection{Energy estimate} \label{ssec : enest}
We restrict in the remainder to the case of negative curvature, the positive case is analogous. Before defining the total energy of the system, we cast the evolution equations into a form where the terms are ordered according to their eventual asymptotic behavior. This reads as follows.
\eq{\label{evolution}\alg{
\partial_\T{g}_{ij}=&-\frac{n}{\sinh(\T)}2 N{\Sigma}_{ij}+\frac{\cosh(\T)}{\sinh(\T)}A
+\frac{n}{\sinh(\T)}B, \\
\partial_\T{\Sigma}_{ij}=&-\frac{\cosh(\T)}{\sinh(\T)}(n-1){\Sigma}_{ij}+\frac{n}{\sinh(\T)} N(-\frac12\Delta_{g,\gamma}(g-\gamma)-\mathring{R}_\gamma(g-\gamma))\\
&+\frac{1}{\sinh(\T)}(\mathcal{L}_{{X}}{\Sigma}_{ij})+\frac{n}{\sinh(\T)}C+\frac{\cosh(\T)}{\sinh(\T)}D.
}}

\begin{lem}
Let $s>\frac n2+1$ and $(g,\Sigma)\in \mathcal B^s_{\delta_g}(\gamma)\times \mathcal B_{\delta_\Sigma}^{s-1}(0)$ for some $\delta_g,\delta_\Sigma>0$ sufficiently small. Then for the perturbation terms, the following estimates hold.
\eq{\label{eq-resc-ev}\alg{
\Ab{A}_s&\leq C(\delta_g,\delta_\Sigma)\Big[ \Ab{\Sigma}_{s-2}^2(\Ab{g-\gamma}_{s}+\Ab{\gamma}_s)\Big]\\
\Ab{B}_s&\leq2\Ab{\nabla X}_s\\
\Ab{C}_s& \leq C(\delta_g,\delta_\Sigma)\Big[\Ab{\Sigma}_s^2+\Ab{N}_s\Ab{\Sigma}^2_s+\Ab{\Sigma}_s^2(\Ab{g-\gamma}_{s}+\Ab{\gamma}_s)+\Ab{N}_s\Ab{g-\gamma}^2_{s+1}\Big]\\
\Ab{D}_s&\leq C(\delta_g,\delta_\Sigma)\Big[\Ab{\Sigma}_{s-2}^2\Ab{\Sigma}_{s}\Big]
}}
\end{lem}
\begin{proof}
The perturbation terms have the explicit form
\eq{\label{ABCD}\alg{
A&=-2(1-nN)(g_{ij}-\gamma_{ij}+\gamma_{ij})\\
B&=\mathcal L_X g_{ij}=\nabla_i X_j+\nabla_jX_i\\
C&=-\nabla_i\nabla_j{N}-2 N{\Sigma}_{ik}{\Sigma}^k_j+(N-\frac{1}{n})({g}_{ij}-\gamma_{ij}+\gamma_{ij})+N J_{ij}\\
D&=-(n-2)(Nn-1){\Sigma}_{ij}.
}}
Note that $J$ is the perturbation term given in lemma \ref{lem-ricc-dec} with the corresponding estimate \eqref{est-J}. Now, the lemma follows from the estimates from lemma \ref{ellipticestimates}, .
\end{proof}
 
\noindent We define the main total energy.
 \begin{defn} Let $s>n/2+1$, then we denote
 \eq{\alg{
 \mathbf E_s(g,\Sigma)&\equiv \Ab{g-\gamma}_{L^2(\gamma)}^2+\sum_{k=0}^{s-1}(-1)^k(\Sigma,\Delta^k_{g,\gamma}\Sigma)_{L^2(g,\ga)}+\frac14\sum_{k=1}^s(-1)^k(g-\gamma,\Delta^k_{g,\gamma}(g-\gamma))_{L^2(g,\ga)}\\
 &\equiv \mathrm{I}+\mathrm{I\!I}+\mathrm{I\!I\!I}.}}
\end{defn}
\begin{rem}Note that the energy is equivalent to the $H^s\times H^{s-1}$-norm of $(g,\Sigma)$.
\end{rem}
\begin{lem} \label{lem : en-est}Let $s>n/2+1$ and $(g,\Sigma)\in H^s\times H^{s-1}$ be a solution to \eqref{resc-evol-1}--\eqref{resc-evol-2}. Then there exists an $\varepsilon>0$ such that for
\eq{
(g,\Si)\in\mcl B^s_{\varepsilon}(\ga)\times\mcl B^{s-1}_{\varepsilon}(0),
}
the estimate 
\eq{\label{en-est}
\partial_{\mathrm T} \mathbf E_s(g,\Sigma)\leq \frac{C(\varepsilon)}{\sinh (\mathrm T)}\mathbf E_s(g,\Sigma)
}
holds in the case of negative curvature of $\ga$. The analogous estimate with $\sinh(\T)$ replaced by $\cosh(\T)$ holds in the case of positive curvature of $\ga$.
\end{lem}

\begin{proof}
We consider the negative curvature case for the proof. The case of positive curvature is analogous. \\
\noindent We take the time derivatives of the three individual terms. 
\eq{\alg{
\partial_{\T}\Ab{g-\gamma}_{L^2(\gamma)}^2&=2\int_M\langle\partial_\T g, (g-\gamma)\rangle_\gamma\mu_\gamma\\
&=2\int_M\langle-\frac{n}{\sinh(\T)}2 N{\Sigma}+\frac{\cosh(\T)}{\sinh(\T)}A
+\frac{n}{\sinh(\T)}B, (g-\gamma)\rangle_\gamma\mu_\gamma\\
&=-\frac{4n}{\sinh(\T)}\int_MN\langle {\Sigma}, (g-\gamma)\rangle_\gamma\mu_\gamma+2\frac{\cosh(\T)}{\sinh(\T)}\int_M\langle A, (g-\gamma)\rangle_\gamma\mu_\gamma\\
&\quad+\frac{2n}{\sinh(\T)}\int_M\langle B, (g-\gamma)\rangle_\gamma\mu_\gamma,
}}
where $A$ and $B$ are given in \eqref{ABCD}.
In particular, we have an estimate of the form

\eq{\alg{
|\partial_{\T}\Ab{g-\gamma}_{L^2(\gamma)}^2|&\leq C\Big[\frac{1}{\sinh(\T)}\Ab{N}_{L^\infty}\mathbf E_{s}(g,\Sigma)+\frac{\cosh(\T)}{\sinh(\T)}\Ab{A}_{L^2}\sqrt{\mathbf E_{s}(g,\Sigma)}\\
&\qquad\qquad+\frac{1}{\sinh(\T)}\Ab{B}_{L^2}\sqrt{\mathbf E_{s}(g,\Sigma)}\Big]
}}
We order the terms according to their appearance in the energy estimate, substitute the expressions for $A$ and $B$.
\eq{\alg{
|\partial_{\T}\Ab{g-\gamma}_{L^2(\gamma)}^2|&\leq C\Big[\frac{1}{\sinh(\T)}n\mathbf E_{s}(g,\Sigma)\\
&+C(\delta_g,\delta_\Sigma)\frac{\cosh(\T)}{\sinh(\T)}\Big[ \Ab{\Sigma}_{L^2}^2(\Ab{g-\gamma}_{s}+\Ab{\gamma}_s)\Big]\sqrt{\mathbf E_{s}(g,\Sigma)}\\
&\qquad\qquad+\frac{2}{\sinh(\T)} \Bigg(C(\delta_g,\delta_\Sigma)\Big[\Ab{{\Sigma}}^3_{s-2}+\Ab{N}_{s-2}\Ab{{\Sigma}}_{s-2}\Ab{g-\gamma}_{s-1}\\
&\qquad\qquad\qquad+\cosh(\T)(1-\frac{2}{n})
(\frac{2}{n}-1)\Ab{\Sigma}^2_{s-3}\Big]\Bigg)\sqrt{\mathbf E_{s}(g,\Sigma)}\Big]
}}
We sort the terms on the RHS into two different categories, where $C$ is a new constant and we make use of the fact that 
$\mathbf E_{s}(g,\Sigma)\leq 1$. Then we obtain
\eq{\label{pt I}\alg{
|\partial_{\T}\Ab{g-\gamma}_{L^2(\gamma)}^2|&\leq \frac{C}{\sinh(\T)}n\mathbf E_{s}(g,\Sigma)+C\frac{\cosh(\T)}{\sinh(\T)}\Ab{\Sigma}_{L^2}^2\sqrt{\mathbf E_{s}(g,\Sigma)}\\
&\qquad\qquad+C\frac{\cosh(\T)}{\sinh(\T)} 
\Ab{\Sigma}^2_{s-3}\sqrt{\mathbf E_{s}(g,\Sigma)}.
}}

\begin{rem}
The two terms on the RHS contribute either to the estimate to be proven or will be absorbed by a large negative term as demonstrated further below.
\end{rem}

\noindent We proceed with the evaluation of the next time derivative.
\eq{\label{dt II}\alg{
&\partial_{\T}\sum_{k=0}^{s-1}(-1)^k( \Sigma,\Delta^k_{g,\gamma}\Sigma)_{L^2(g,\ga)}\\
&=\sum_{k=0}^{s-1}(-1)^k\Bigg\{\int_M 2\langle\partial_T\Sigma,\Delta_{g,\gamma}^k\Sigma\rangle+\langle\Sigma,[\partial_\T,\Delta_{g,\gamma}^k]\Sigma\rangle\mu_g+\int_M\langle\Sigma,\Delta_{g,\gamma}^k\Sigma\rangle\partial_\T\mu_{g}\Bigg\}
}}
\eq{\nonumber\alg{
&=\sum_{k=0}^{s-1}(-1)^k\Bigg\{-2\frac{\cosh(\T)}{\sinh(\T)}(n-1)\int_M \langle{\Sigma},\Delta_{g,\gamma}^k\Sigma\rangle \mu_{g}\\
&\qquad\qquad\qquad\qquad-\frac{1}{\sinh(\T)}\int_M \langle N(\Delta_{g,\gamma}(g-\gamma)),\Delta_{g,\gamma}^k\Sigma\rangle \mu_{g}\\
&\qquad\qquad-\frac{2}{\sinh(\T)}\int_M \langle N( \mathring{R}_\gamma(g-\gamma)),\Delta_{g,\gamma}^k\Sigma\rangle \mu_{g}+\frac{2}{\sinh(\T)}\int_M \langle(\mathcal{L}_{{X}}{\Sigma}),\Delta_{g,\gamma}^k\Sigma\rangle \mu_{g}\\
&\qquad\qquad\qquad\qquad+\frac{2}{\sinh(\T)}\int_M \langle C,\Delta_{g,\gamma}^k\Sigma\rangle \mu_{g}+2\frac{\cosh(\T)}{\sinh(\T)}\int_M \langle D,\Delta_{g,\gamma}^k\Sigma\rangle \mu_{g}\\
&\qquad\qquad\qquad+\int_M\langle\Sigma,[\partial_\T,\Delta_{g,\gamma}^k]\Sigma\rangle\mu_{g}+\int_M\langle\Sigma,\Delta_{g,\gamma}^k\Sigma\rangle\partial_\T\mu_{g}\Bigg\}
}}
Before we further evaluate the previous term, we compute the time derivative of the last term of the energy and evaluate both terms in combination.\\

\noindent The time derivative of the last term reads
\eq{\label{dt III}\alg{
&\partial_{\T}\frac14\sum_{k=1}^s(-1)^k\langle g-\gamma,\Delta^k_{g,\gamma}(g-\gamma)\rangle_{L^2(g,\ga)}\\
&=\frac14\sum_{k=1}^{s}(-1)^k\Bigg\{\int_M 2\langle\partial_T(g-\gamma),\Delta_{g,\gamma}^k(g-\gamma)\rangle+\langle g-\gamma,[\partial_\T,\Delta_{g,\gamma}^k](g-\gamma)\rangle\mu_g\\
&\qquad\qquad\qquad+\int_M\langle(g-\gamma),\Delta_{g,\gamma}^k(g-\gamma)\rangle\partial_\T\mu_{g}\Bigg\}\\
&=\frac14\sum_{k=1}^{s}(-1)^k\Bigg\{-\frac{4}{\sinh(\T)}\int_M N\langle \Sigma,\Delta_{g,\gamma}^k(g-\gamma)\rangle\mu_g+2\frac{\cosh(\T)}{\sinh(\T)}\int_M \langle A,\Delta_{g,\gamma}^k(g-\gamma)\rangle\mu_g\\
&\qquad\qquad\qquad+\frac{2}{\sinh(\T)}\int_M \langle B,\Delta_{g,\gamma}^k(g-\gamma)\rangle\mu_g+\int_M\langle g-\gamma,[\partial_\T,\Delta_{g,\gamma}^k](g-\gamma)\rangle\mu_g\\
&\qquad\qquad\qquad+\int_M\langle(g-\gamma),\Delta_{g,\gamma}^k(g-\gamma)\rangle\partial_\T\mu_{g}\Bigg\}.
}}
In both previous computations there are commutator terms arising. We do an intermediate discussion of those terms in the following. The commutator operator can also be written as
\eq{\label{commutator}[\partial_\T ,\Delta_{g,\gamma}^k](h)=(\partial_\T\Delta_{g,\gamma}^k)(h)=\sum_{l=0}^{k-1}(\Delta_{g,\gamma}^l\circ (\partial_\T\Delta_{g,\gamma})\circ\Delta_{g,\gamma}^{l-k-1})(h)}
for some 2-tensor $h$. 
Recall that the Laplacian appearing here has the local formula
\eq{\Delta_{g,\gamma}h_{ij}=\frac{1}{\mu_g}\nabla[\gamma]_m(g^{mn}\mu_g\nabla[\gamma]_n h_{ij})}
This shows that the variation of the operator with respect to the metric can be written schematically as
\eq{\label{schema}\alg{(\partial_\T\Delta_{g,\gamma})(h)=&
\partial_\T g *\Delta_{g,\gamma}h+\frac{1}{\mu_g}\nabla[\gamma]_m(\mu_g\partial_\T g*\nabla[\gamma]_n h)\\
=&\partial_\T g*\nabla^2[\gamma]h+\partial_\T g*\nabla[\gamma]g*\nabla[\gamma]h
+\nabla[\gamma]\partial_\T g*\nabla[\gamma]h}}
where $*$ is Hamilton's notation of a combination of tensor products with contractions with respect to $g$.
Therefore, by \eqref{commutator} and \eqref{schema} and by suitable integration by parts, one can see that
\eq{\label{comm 1}\alg{
|\sum_{k=0}^{s-1}(-1)^k\int_M\langle\Sigma,[\partial_\T,\Delta_{g,\gamma}^k]\Sigma\rangle\mu_{g}|
\leq C\left\|\Sigma\right\|^2_{H^{s-1}(g)}\left\|\partial_\T g\right\|_{H^{s-2}(g)}
}}
and similarly,
\eq{\label{comm 2}\alg{
|\sum_{k=1}^{s}(-1)^k\int_M\langle g-\gamma,[\partial_\T,\Delta_{g,\gamma}^k](g-\gamma)\rangle\mu_{g}|
\leq C\left\|g-\gamma\right\|^2_{H^{s}(g)}\left\|\partial_\T g\right\|_{H^{s-1}(g)}
}}
and for $\partial_\T g$, we have good estimates which make this terms to be of higher order.\\

\noindent Before we continue, we note the following estimate for the norm of the time derivative of $g$, which follows straightforward from the previous estimates. We have
\eq{\label{est-dtg}\alg{
\Abkg{\p \T g}{s-1}&\leq C\Bigg[\frac{1}{\sinh(\T)}\Abkg{N}{s-1}\Abkg{\Si}{s-1}\\
&\qquad+\frac{\cosh(\T)}{\sinh(\T)}\Abkg{\Si}{s-1}^2\Big(1+\Abkg{g-\ga}s\Big)\\
&\qquad+\frac{1}{\sinh(\T)}\Big(\Abkg{\Si}{s-1}^3+\Abkg{N}{s-1}\Abkg{\Si}{s-1}\Abkg{g-\ga}{s-1} \Big)\Bigg].
}}
\noindent We have evaluated both commutator terms arising in \eqref{dt II} and \eqref{dt III} and proceed by combining the terms on the corresponding RHS into two different classes. We rearrange the terms.

\eq{\label{pt II+III}\alg{
\partial_\T (\mathrm{I\!I}+\mathrm{I\!I\!I})=&-2\frac{\cosh(\T)}{\sinh(\T)}(n-1)\sum_{k=0}^{s-1}(-1)^k\int_M \langle{\Sigma},\Delta_{g,\gamma}^k\Sigma\rangle \mu_{g}\\
&-\sum_{k=0}^{s-1}(-1)^k\frac{2}{\sinh(\T)}\int_M \langle N(\Delta_{g,\gamma}(g-\gamma)),\Delta_{g,\gamma}^k\Sigma\rangle \mu_{g}\\
&-\sum_{k=1}^{s}(-1)^k\frac{2}{\sinh(\T)}\int_M N\langle \Sigma,\Delta_{g,\gamma}^k(g-\gamma)\rangle\mu_g\\
&+\sum_{k=0}^{s-1}(-1)^k\Bigg\{-\frac{2}{\sinh(\T)}\int_M \langle N(2\mathring{R}_\gamma(g-\gamma)),\Delta_{g,\gamma}^k\Sigma\rangle \mu_{g}\\
&\qquad\qquad\qquad+\frac{2}{\sinh(\T)}\int_M \langle(\mathcal{L}_{{X}}{\Sigma}),\Delta_{g,\gamma}^k\Sigma\rangle \mu_{g}\\
&\qquad\qquad\qquad+\frac{2}{\sinh(\T)}\int_M \langle C,\Delta_{g,\gamma}^k\Sigma\rangle \mu_{g}\\
&\qquad\qquad\qquad+2\frac{\cosh(\T)}{\sinh(\T)}\int_M \langle D,\Delta_{g,\gamma}^k\Sigma\rangle \mu_{g}\\
&\qquad\qquad\qquad+\int_M\langle\Sigma,[\partial_\T,\Delta_{g,\gamma}^k]\Sigma\rangle\mu_{g}+\int_M\langle\Sigma,\Delta_{g,\gamma}^k\Sigma\rangle\partial_\T\mu_{g}\Bigg\}\\
&+\frac14\sum_{k=1}^{s}(-1)^k\Bigg\{2\frac{\cosh(\T)}{\sinh(\T)}\int_M \langle A,\Delta_{g,\gamma}^k(g-\gamma)\rangle\mu_g\\
&\qquad\qquad\qquad\qquad+\frac{2}{\sinh(\T)}\int_M \langle B,\Delta_{g,\gamma}^k(g-\gamma)\rangle\mu_g\\
&\qquad\qquad\qquad\qquad+\int_M\langle g-\gamma,[\partial_\T,\Delta_{g,\gamma}^k](g-\gamma)\rangle\mu_g\\
&\qquad\qquad\qquad\qquad+\int_M\langle(g-\gamma),\Delta_{g,\gamma}^k(g-\gamma)\rangle\partial_\T\mu_{g}\Bigg\}\\
\leq&-\underbrace{2\frac{\cosh(\T)}{\sinh(\T)}(n-1)\sum_{k=0}^{s-1}(-1)^k\int_M \langle{\Sigma},\Delta_{g,\gamma}^k\Sigma\rangle \mu_{g}}_{(*)} +\underbrace{C\sqrt{\mathbf E_{s}(g,\Sigma)}\Ab{\Sigma}^2_{H^{s-1}}}_{(**)}\\
&\quad+\frac{C}{\sinh (\T)}\mathbf E_{s}(g,\Sigma)
}}
We still have to justify the last inequality. Therefore we have to analyze all terms on the right hand side above and estimate them by one of the three terms given in the last two lines of \eqref{pt II+III}. Before we do this, we first argue, why this estimate implies the result. The first term on the right hand side, ($*$), has a negative sign and its absolut value bounds the $H^{s-1}$-norm of $\Sigma$ up to a multiplicative positive constant from above. Therefore choosing $\sqrt{\mathbf E_{s}(g,\Sigma)}$ sufficiently small by choosing the $g$ close to $\gamma$ we can ensure that the second term, $(**)$, is always bounded from above by the absolut value of the first term and thereby the sum of both terms is negative and can be estimated from above by 0 yielding the desired estimate.\\
To complete the proof we need to justify the last estimate in \eqref{pt II+III}. We proceed term by term.\\

\noindent The second and third line of the RHS of \eqref{pt II+III} contain leading terms of too high regularity to close the estimate - these terms cancel pairwise using integration by parts. The resulting term can then be estimated as follows.
\eq{\alg{
&\Big|-\frac{2}{\sinh(\T)}\sum_{k=1}^{s-1}(-1)^k\int \langle[\Delta^k_{g,\gamma}(g-\gamma),N]\Delta_{g,\gamma}(g-\gamma),\Sigma\rangle \mu_g\Big|\\
&\qquad\qquad\qquad\qquad\qquad\qquad\qquad\qquad\qquad\leq \frac{C}{\sinh(\T)} \Abkg{N}s\Abkg{g-\gamma}s\Abkg{\Si}{s-1}
}}

\noindent This term contributes to the last term on the RHS of \eqref{pt II+III}.\\

\noindent The term in the fourth line is evaluated using $\overset{\,\circ}R_\gamma(g-\gamma)_{ij}=R_{ikjl}(\gamma)(g-\gamma)^{kl}$. We estimate
\eq{
\Big|\frac{-4}{\sinh(\T)}\sum_{k=1}^{s-1}\int \langle N \overset{\,\circ}R_\gamma(g-\gamma),\Delta^k_{g,\gamma}\Si\rangle\mu_g  \Big|\leq \frac{C}{\sinh(\T)}\Abkg N s\Abkg {g-\ga}s\Abkg{\Si}{s-1},
}
which contributes again to the last term on the right hand side of \eqref{pt II+III}.\\

\noindent To evaluate the term in the fourth line one needs to observe the symmetry when using the integration by parts. This yields an estimate of the form
\eq{
\Big| \frac{2}{\sinh(\T)}\sum_{k=1}^{s-1}\int_M \langle \mcl L_X\Si,\Delta^k_{g,\gamma}\Si\rangle\mu_g\Big|\leq \frac{C}{\sinh(\T)}\Abkg{X}s\mathbf E_s(g,\Si).
}

\noindent The sixth and seventh line can be evaluated straightforwardly using the estimates for the $H^s$ norms of $C$ and $D$ as defined in \eqref{ABCD}. As the second term containing $D$ does not contain a good time factor it is important to note that this term contains a factor quadratic in the $H^{s-1}$-norm of $\Sigma$. 
Precisely, we have
\eq{\alg{
&\Big|\sum_{k=1}^{s-1}\frac{2}{\sinh(\T)}\int_M \langle C,\Delta_{g,\gamma}^k\Sigma\rangle \mu_{g}+2\frac{\cosh(\T)}{\sinh(\T)}\int_M \langle D,\Delta_{g,\gamma}^k\Sigma\rangle \mu_{g}\Big|\\
&\qquad\leq\frac{2}{\sinh(\T)}\Abkg{C}{s-1}\Abkg{\Si}{s-1}+2\frac{\cosh(\T)}{\sinh(\T)}\Abkg{D}{s-1}\Abkg{\Si}{s-1}\\
&\qquad\leq\frac{2C}{\sinh(\T)}\Abkg{\Si}{s-1}\Big[\mathbf E_s(g,\Si)(1+\Abkg N s+\sqrt{\mathbf E_s(g,\Si)})\Big]\\
&\qquad\quad+2C\frac{\cosh(\T)}{\sinh(\T)}\Abkg{\Si}{s-1}^4,
}}
which contributes to the second term and the third term on the RHS. Note, that by smallness of $\Si$ we can absorb terms with higher exponents into the explicitly given one.\\

\noindent The first term in the sixth line has been evaluated in \eqref{comm 1} and clearly is quadratic in the $H^{s-1}$-norm of $\Sigma$. The other factors yield the necessary energy factor. \\
We recall the estimate
\eq{
\Big|\sum_{k=1}^{s-1}\int_M\langle\Sigma,[\partial_\T,\Delta_{g,\gamma}^k]\Sigma\rangle\mu_{g}\Big|\leq C\left\|\Sigma\right\|^2_{H^{s-1}(g)}\left\|\partial_\T g\right\|_{H^{s-2}(g)},
}
which in combination with the estimate for the last factor, \eqref{est-dtg} shows that the terms can be estimated as claimed contributing to the second and third term on the RHS of \eqref{pt II+III}.

The second term in the eigth line is determined by the time derivative of the volume form. We use the identity
\eq{
\p \T\mu_g=\frac12 g^{ij}\p \T g_{ij}\mu_g =n\Big(\frac{\cosh(\T)}{\sinh(\T)}(N-\frac1n)+\frac{1}{\sinh(\T)}\na_iX^i\Big)\mu_g
}
and estimate
\eq{
\Big|\sum_{k=1}^{s-1}\int_M\langle\Sigma,\Delta_{g,\gamma}^k\Sigma\rangle\partial_\T\mu_{g}\Big|\leq C\Bigg[\frac{\cosh(\T)}{\sinh(\T)}\Abkg{N-\frac1n}{s-1}+\frac1{\sinh (\T)}\Abkg{X}{s}\Bigg]\Abkg{\Si}{s-1}^2.
}

\noindent These terms can be estimated by the second and third term in \eqref{pt II+III}.\\

\noindent The terms in the lines nine to twelve of \eqref{pt II+III} are of the same type as the terms above. We briefly describe the way to estimate them in the following.

\noindent The ninth and tenth line can be estimated as follows
\eq{\alg{
&\Big|\sum_{k=1}^{s}(-1)^k\frac{\cosh(\T)}{\sinh(\T)}\int_M \langle A,\Delta_{g,\gamma}^k(g-\gamma)\rangle\mu_g+\frac{1}{\sinh(\T)}\int_M \langle B,\Delta_{g,\gamma}^k(g-\gamma)\rangle\mu_g\Big|\\
&\quad\leq C\Big[\frac{\cosh(\T)}{\sinh(\T)}\Abkg{A}s\Abkg{g-\ga}s+\frac{1}{\sinh(\T)}\Abkg{B}s\Abkg{g-\ga}s\Big].
}}
In combination with the estimates for $A$ and $B$ we can estimate these terms by the terms on the right hand side of \eqref{pt II+III}. The commutator term in the eleventh line has been already evaluated in \eqref{comm 2}, which yields an estimate of the form as treated above. Finally, the last line simply contains the time derivative of the metric and can be treated as the corresponding term above.\\

\noindent
We have analyzed all relevant terms in the estimate, which in combination with the argument following \eqref{pt II+III} finishes the proof.

\end{proof}


\subsection{Improved decay} \label{ssec : imprdec}
We proceed by deriving an energy estimate for the Sobolev norm of the trace free part of the second fundamental form in one order of regularity below the maximal regularity.
This estimate holds under the condition that boundedness of the total energy of maximal regularity is given. The structure of this estimate is such that it eventually leads to decay of the trace free part of the second fundamental form in this regularity.

\begin{lem}\label{lem : impr dec}
Let 
\eq{
\mathbf H_{s-2}(\Sigma)\equiv\sum_{k=0}^{s-2}(-1)^k(\Sigma,\Delta^k_{g,\gamma}\Sigma)_{L^2(g,\ga)}.
}
Assume for some $0<\varepsilon<1$ 
\eq{
\mathbf E_{s}(g,\Si)\leq \varepsilon.
}
Then
\eq{\alg{
\p \T\mathbf H_{s-2}(\Si)&\leq -2(n-1)\frac{\cosh(\T)}{\sinh(\T)}\mathbf H_{s-2}(\Si)+C\frac{\cosh(\T)}{\sinh(\T)}\Big(\mathbf H_{s-2}(\Si)\Big)^2\\
&\quad+\frac{C}{\sinh(\T)}\Big[\Ab{X}_{H^{s-1}(g)}\mathbf H_{s-2}(\Si)\\
&\qquad\qquad\qquad+(1+\Abkg N {s-1}+\sqrt{\varepsilon})\sqrt{\varepsilon}\sqrt{\mathbf H_{s-2}(\Si)}\Big].
}}
\end{lem}

\begin{proof}
The proof follows from the identical computations in the proof of lemma \ref{lem : en-est}. The higher order term containing the Laplacian is estimated using integration by parts and H\"older's estimate.
\end{proof}

\noindent In combination with the elliptic estimates for lapse and shift \eqref{ellipticestimates} we obtain the following corollary.

\begin{cor}
Under the assumptions of lemma \ref{lem : impr dec} the following estimate holds.

\eq{\alg{
\p \T\mathbf H_{s-2}(\Si)&\leq -2(n-1)\frac{\cosh(\T)}{\sinh(\T)}\mathbf H_{s-2}(\Si)+C\frac{\cosh(\T)}{\sinh(\T)}\Big(\mathbf H_{s-2}(\Si)\Big)^2\\
&\quad+\frac{C}{\sinh(\T)}\Big[\Big(\mathbf H_{s-2}(\Si)+\sqrt{\varepsilon}(C+\mathbf H_{s-2}(\Si))\Big)\Big(\mathbf H_{s-2}(\Si)\Big)^{3/2}\\
&\qquad\qquad\qquad+(1+\mathbf{H}_{s-3}(\Si)+\sqrt{\varepsilon})\sqrt{\varepsilon}\sqrt{\mathbf H_{s-2}(\Si)}\Big].
}}
As in particular 
\eq{
\mathbf{H}_{s-2}(\Si)<\varepsilon
}
this simplifies to
\eq{\alg{
\p \T\sqrt{\mathbf H_{s-2}(\Si)}&\leq -(n-1)\frac{\cosh(\T)}{\sinh(\T)}\sqrt{\mathbf H_{s-2}(\Si)}+C\frac{\cosh(\T)}{\sinh(\T)}\Big(\mathbf H_{s-2}(\Si)\Big)^{3/2}+\frac{C\sqrt{\varepsilon}}{\sinh(\T)}.
}}
\end{cor}

\noindent In the following we deduce an estimate for $\mathbf{H}_{s-2}(\Si)$.

\begin{lem}
Under the same assumptions as in the previous lemmas and if $n>2$, 
\eq{
\sqrt{\mathbf{H}_{s-2}(\Si)}\leq \frac{C\sqrt{\varepsilon}}{\sinh(\T)}.
}
If $n=2$, the estimate is
\eq{
\sqrt{\mathbf{H}_{s-2}(\Si)}\leq \frac{C\sqrt{\varepsilon}}{\sinh^{1/2}(\T)}.
}
\end{lem}
\begin{proof}If $n>2$, we have
\eq{\alg{
\p \T\sqrt{\mathbf H_{s-2}(\Si)}&\leq -\frac{\cosh(\T)}{\sinh(\T)}\sqrt{\mathbf H_{s-2}(\Si)}+\frac{C\sqrt{\varepsilon}}{\sinh(\T)}.
}}
and if $n=2$, we have
\eq{\alg{
\p \T\sqrt{\mathbf H_{s-2}(\Si)}&\leq -\frac{1}{2}\frac{\cosh(\T)}{\sinh(\T)}\sqrt{\mathbf H_{s-2}(\Si)}+\frac{C\sqrt{\varepsilon}}{\sinh(\T)},
}}
provided that $\varepsilon$ is small enough. These differential inequalities immediately imply the desired decay.
\end{proof}
\begin{lem}\label{shiftdecay0}
Under the same assumptions as above, the shift vector admits the estimates
\begin{align}
\left\|X\right\|_{s}\leq C\varepsilon\cdot \sinh^{-1}(\T)
\end{align}
if $n>2$ and
\begin{align}
\left\|X\right\|_{s}\leq  C\varepsilon\cdot\sinh^{-1/2}(\T),
\end{align}
if $n=2$.
\end{lem}
\begin{proof}
This follows from the previous lemma and \eqref{Xestimate}.
\end{proof}
\begin{rem}
From the decay of $X$ and $\Sigma$, we also get an exponential decay of $\partial_\T g$ in the $H^{s-2}$-norm, i.e.\ the rescaled metrics converge to a limit metric in $H^{s-2}$ as $\T\to \infty$. This follows from \eqref{evolution} and \eqref{eq-resc-ev}.
\end{rem}


\subsection{Proof of the main theorem}\label{ssec : proof}
Using the previous lemmas we are now able to state the proof of Theorem \ref{main-thm-neg} and Theorem \ref{main-thm}. We give the explicit proof for Theorem \ref{main-thm-neg}, the positive case is analogous.

\begin{proof}[Proof of theorem \ref{main-thm-neg}]
Let $\varepsilon>0$ be fixed. Before we consider $\delta$-small CMCSH-initial data at initial time $\sqrt{2}$,
assume we start with arbitrary initial data. 
By Theorem \ref{CMCthm} we have a small $H^{s'}\times H^{s'-1}$ neighbourhood $\mathcal{V}$ in the set of arbitrary initial data such that the MGHD of any data in $\mathcal{V}$ admits a hypersurface of constant mean curvature $-\sqrt{2}$ and the induced data $(g_0,\Sigma_0)$ stays in a small $H^{s}\times H^{s-1}$ neighbourhood $\mathcal{U}$ of the initial data of the background.
 By Theorem \ref{slicethm} we can pull back the data along a diffeomorphism $\varphi$ such that $(\varphi^*g_0,\varphi^*\Sigma_0)$ satisfies the CMCSH gauge and is $\delta$-close to the initial data of the background solution. From now on, we let the data evolve under the Einstein-flow and the corresponding solution will be isometric to the MGHD of the initial data we started with.
\\
Without loss of generality, assume that $\varepsilon$ is so small that \eqref{en-est} holds as long as $E_s(g(t),\Sigma(t))<2\varepsilon$.
Let $\T_{\mathrm{max}}$ be the maximal existence time of the solution and suppose that $\T_{\max}<\infty$.
 By the Gronwall inequality,
\eq{
\alg{
E_s(g(\T_{\mathrm{max}}),\Sigma(\T_{\mathrm{max}}))\leq &e^{C(2\varepsilon)\int_{\sqrt{2}}^{\T_{\mathrm{max}}}
\frac{d\T}{\sinh(\T)}}\cdot E_s(g(\sqrt{2}),\Sigma(\sqrt{2}))\\
\leq &e^{C(2\varepsilon)\int_{\sqrt{2}}^{\infty}
\frac{d\T}{\sinh(\T)}}\cdot E_s(g(\sqrt{2}),\Sigma(\sqrt{2}))
}
}
which shows that the left hand side is bounded by
some arbitrarily small $\varepsilon_1$ supposed that $\delta$ was chosen small enough. Due to local existence, this contradicts the maximality of the existence time. Therefore, $\T_{\mathrm{max}}=\infty$ and since the energy is equivalent to the $H^s\times H^{s-1}$-norm of the data, we also obtained the desired bound on the solution.

\noindent To complete the proof we show global hyperbolicity and future- and null geodesic completeness.
For this purpose, some properties of the lapse, the shift, the second fundamental form and the family of Riemannian metrics have to be checked. At first, by the estimate of the main energy, lemma \ref{ellipticestimates} and lemma \ref{shiftdecay0}, we can choose for any $\varepsilon>0$ a neighbourhood of the initial data such that
\eq{\alg{
\left\|g-\gamma\right\|_{C^0(\gamma)}&\leq\varepsilon,\qquad
\left\|N-\frac{1}{n}\right\|_{C^0}\leq\varepsilon\sinh^{-1}(\T),\qquad
\left\|\nabla N\right\|_{C^0(g)}\leq\varepsilon\sinh^{-1}(\T),\\
\left\|\Sigma\right\|_{C^0(g)}&\leq \varepsilon\sinh^{-1/2}(\T), \qquad \left\|X\right\|_{C^0(g)}\leq\varepsilon \sinh^{-1/2}(\T)
}
}
for all $\T\geq \T_0$. Recall that the non-rescaled metrics $\tilde{g}$ satisfy
$\tilde{g}=s(\tau)^{-1}g=\sinh^2(\T)g$ which implies that for all tangent vectors $v$, we have
\begin{align}\label{metricbound}\gamma_{ij}v^iv^j\leq Cg_{ij}v^iv^j\leq C \sinh^{-2}(\T)\tilde{g}_{ij}v^iv^j
\end{align}
which implies that the metrics $\tilde{g}$ are uniformly bounded from below by $\gamma$. For $\tilde{\Sigma}=s^{-1/2}(\tau)\Sigma$, we have
\begin{align}\label{sigmadecay}|\tilde{\Sigma}|_{\tilde{g}}\leq \sinh^{-1}(\T)\cdot|\Sigma|_{g}\leq \sinh^{-3/2}(\T)\cdot\varepsilon
\end{align}
For the original shift vector $\tilde{X}=\sinh(\T)X$, we have
\begin{align}
\left\|\tilde{X}\right\|_{C^0(\tilde{g})}\leq\varepsilon \sinh^{3/2}(\T).
\end{align}
However, this is the shift vector from the CMC-gauge and  we have to define the shift vector which corresponds to the time function $\T$:
\begin{align}\bar{X}^{\flat}:=\left\langle \frac{\partial}{\partial \T},.\right\rangle=\frac{n}{\sinh^2(\T)}\left\langle \frac{\partial}{\partial \tau},.\right\rangle=\frac{n}{\sinh^2(\T)}\tilde{X}^{\flat}.
\end{align}
Here, $\bar{X}^{\flat}$ and $\tilde{X}^{\flat}$ are the $1$-forms which are equivalent to $\bar{X}$ and $\tilde{X}$ via the metric $\tilde{g}$. Now we immediately get
\begin{align}\label{shiftdecay}
\left\|\bar{X}\right\|_{\tilde{g}}\leq\varepsilon n\cdot\sinh^{-1/2}(\T)
\end{align}
For the original Lapse function $\tilde{N}$ we have $\tilde{N}=s(\tau)^{-1}N=\sinh^2(\T)N$, which yields
\begin{align}
\left\|\sinh^{-2}(\T)\tilde{N}-\frac{1}{n}\right\|_{C^0}\leq\varepsilon\sinh^{-1}(\T),\qquad
\left\|\tilde{\nabla}\tilde{N}\right\|_{C^0(\tilde{g})}\leq \varepsilon .
\end{align}
However, as for the shift vector, $\tilde{N}$ was obtained by the $CMC$-gauge. We now compute the lapse-function $\bar{N}$ according to the time $\T$. We have
\eq{
\alg{\bar{N}^2:=&-\left\langle \frac{\partial}{\partial \T}, \frac{\partial}{\partial \T}\right\rangle+|\bar{X}|^2_{\tilde{g}}\\=&-\left(\frac{n}{\sinh^2(\T)}\right)^2\left(\left\langle \frac{\partial}{\partial \tau}, \frac{\partial}{\partial \tau}\right\rangle-|\tilde{X}|^2_{\tilde{g}}\right)=\left(\frac{n}{\sinh^2(\T)}\right)^2\tilde{N}^2
}
}
which immediately implies 
\begin{align}\label{lapsedecay}
\left\|\bar{N}-1\right\|_{C^0}\leq\varepsilon\cdot n\sinh^{-1}(\T),\qquad
\left\|\tilde{\nabla}\bar{N}\right\|_{C^0(\tilde{g})}\leq \varepsilon\cdot n \sinh^{-2}(\T).
\end{align}
By \cite[Theorem 2.1]{CBC02}, global hyperbolicity follows from \eqref{metricbound}, \eqref{lapsedecay} and \eqref{shiftdecay}.
Causal completeness of the solutions follows from \eqref{sigmadecay} and \eqref{lapsedecay} by using \cite[Theorem 3.2 and Corollary 3.3]{CBC02}.
\end{proof}


\vspace{0.1cm}
\noindent
\textsc{David Fajman\\
Faculty of Physics, University of Vienna,\\
Boltzmanngasse 5, 1090 Vienna, Austria}\\ 
\vspace{-0.4cm}\\
\texttt{David.Fajman@univie.ac.at}\\

\noindent
\textsc{Klaus Kr\"oncke\\
Department of Mathematics, University of Hamburg ,\\
Bundesstrasse 55, 20146 Hamburg, Germany}\\ 
\vspace{-0.4cm}\\
\texttt{klaus.kroencke@uni-hamburg.de}\\

\end{document}